%% file: main.tex
\documentclass[10pt]{article}
\usepackage[english]{babel}

\usepackage[margin=1in]{geometry}

\usepackage{mathtools}
\usepackage{amsmath}
\allowdisplaybreaks  
\usepackage{subcaption}

\newcommand{\Value}[2]{V_{#1}^{#2}} 

\usepackage[nocompress]{cite}

\usepackage{amsmath}
\usepackage{graphicx}
\allowdisplaybreaks
\usepackage[colorlinks=true, allcolors=blue]{hyperref}
\input{style}
\usepackage{smile}

\title{\vspace{-0.5in}Robust Markov Decision Processes on Continuous State Spaces}
\author{
Mengmeng Li \\ 
\small EPFL \\
\small \texttt{mengmeng.li@epfl.ch}
\and
Yifan Hu \\
\small Rutgers University \\
\small \texttt{yifan.hu@rutgers.edu}
\and
Daniel Kuhn \\
\small EPFL \\
\small \texttt{daniel.kuhn@epfl.ch}
\and
Yan Li\footnote{Correspondence to: Yan Li \texttt{<yan.li@tamu.edu>}.} \\
\small Texas A\&M University \\
\small \texttt{yan.li@tamu.edu}
}
\date{}

\begin{document}
\maketitle
\vspace{-0.2in}
\begin{abstract}

We study infinite-horizon robust Markov decision processes (MDPs) on continuous state spaces with structured rectangular ambiguity set. 
The proposed ambiguity set falls within the convex hull of unknown generating kernels. 
We utilize the dynamic formulation of the corresponding robust MDPs, and subsequently introduce a stochastic first-order method for robust policy evaluation.
We establish its high probability convergence to the robust value function, which in turn leads to an  $\tilde{\cO}(1/\epsilon^2)$ sample complexity.
This high probability accuracy certificate is then used in an approximate policy iteration method that finds an $\epsilon$-optimal policy with $\tilde{\cO}(1/\epsilon^2)$ samples. 
The obtained sample complexities for both robust policy evaluation and optimization appear to be new for robust MDPs with continuous state~spaces. 
Of independent interest, the proposed method is also directly applicable to zero-sum Markov games, which seems to strictly improve the existing sample complexities for continuous state~spaces. 
\end{abstract}

\input{intro}
\input{robust_mdp}
\input{robust_eval_outer}
\input{td_nature}
\input{robust_eval_meta}
\input{robust_opt}
\input{outro}

\bibliographystyle{plain}
\bibliography{references}
\input{appendix}

\end{document}

%% file: style.tex
\usepackage{algorithm}
\usepackage{algorithmic}
\usepackage{hyperref}
\usepackage{xcolor}

\usepackage[numbers]{natbib}
\setcitestyle{number}

\newcommand{\yan}[1]{{\color{black} #1}}

\newcommand{\revise}[1]{{\color{black} #1}}

\definecolor{violet}{RGB}{197, 27, 138}

\newcommand{\revm}[1]{{\color{black} #1}}
\usepackage[bottom]{footmisc}
\usepackage{dsfont} 
\usepackage{bbm} 
\usepackage{footnote}
\providecommand{\Call}[2]{\textsc{#1}(#2)} 

\usepackage{comment}
\usepackage{enumerate}

\usepackage{multicol}

\usepackage{tcolorbox}
\usepackage{amsthm}
\usepackage{amsmath}
\usepackage{amssymb}
\usepackage{algorithm}
\newtheorem{theorem}{Theorem}[section]
\newtheorem{lemma}[theorem]{Lemma}
\newtheorem{proposition}[theorem]{Proposition}
\newtheorem{remark}{Remark}
\newtheorem{definition}{Definition}

\newtheorem{assumption}{Assumption}
\newtheorem{condition}{Condition}
\newtheorem{example}{Example}

\DeclareMathOperator*{\argmin}{argmin}
\DeclareMathOperator*{\argmax}{argmax}

\newcommand{\R}{\mathbb{R}}
\newcommand{\mac}{\mathcal}

\newcommand{\diff}{\mathrm{d}}
\newcommand{\opt}{^\star}

\newcommand{\mb}{\mathbb}

\newcommand{\EE}{\mathbb{E}}
\newcommand{\PP}{\mathbb{P}}

\newcommand{\rw}{\mathrm{w}}

\usepackage{enumitem}
\newcommand{\Qsvp}[4]{(#4(\cdot|#2))^{\!\top} Q_{#1}(#2)\,#3(\cdot|#2)}
\newcommand{\hatQsvp}[4]{(#4(\cdot|#2))^{\!\top} \hat Q_{#1}(#2)\,#3(\cdot|#2)}
\newcommand{\Qwsvp}[5]{(#5(\cdot|#3))^{\!\top} Q^{#1}_{#2}(#3)\,#4(\cdot|#3)}

\graphicspath{{figures/}}

\allowdisplaybreaks

%% file: intro.tex

\section{Introduction}

We consider an infinite-horizon robust Markov decision process (MDP), denoted by  $(\mathcal{S}, \mathcal{A}, \cP, r, \gamma)$, comprising a \revm{closed} Borel measurable state space $\mathcal{S}\subseteq \R^S$,
a finite action space~$\cA$, a continuous reward-per-stage function $r: \mathcal{S} \times \mathcal{A} \rightarrow \mathbb{R}$, and an ambiguity set $\cP\subseteq\cP_0$ of transition kernels $P: \mathcal{S} \times \mathcal{A} \rightarrow \Delta_{\cS}$, and a discount factor $\gamma\in(0,1)$.
\revm{Here, $\cP_0$ is the set of all transition kernels $P$ such that  $P(\cB|s,a)$ is a Borel-measurable function of $s $  for every $a\in\cA$ and every Borel set $\cB\subseteq\cS$, and $\Delta_{\cS}$ denotes the set of probability measures over $\cS$.}
We assume that 
$r(s,a) \in [0,1]$ for all $(s,a) \in \cS \times \cA$.
Let~$\Pi$ denote the set of all randomized stationary Borel measurable policies $\pi:\cS\to\Delta_{\cA}$. 
For any $\pi \in \Pi$ and $P \in \cP$,
the value function $V_\pi^P: \mathcal{S} \rightarrow \mathbb{R}$ is defined through
\begin{align}\label{def_nonrobust_value}
    \Value{\pi}{P}(s)=\mb E^P_\pi\left[\sum_{t=0}^{\infty} \gamma^t r(S_t, A_t) \mid S_0=s\right],
\end{align} 
where $\cbr{(S_t, A_t)}_{t \geq 0}$ is the Markov process with its law induced by $\pi$ and $P$ given $S_0 = s$ (\emph{cf}.\ Ionescu Tulcea Theorem \cite{tulcea1949mesures}),
and $\EE^P_\pi$ denotes the corresponding expectation. 
The robust value function  is then defined~as
\begin{align}\label{def_robust_value}
    V_\pi(s) = 
    \inf_{P \in \cP }  \Value{\pi}{P}(s).
\end{align}
We are interested in finding the optimal policy $\pi\opt$ of 
\begin{align}\label{def_rmdp_general}
    \sup_{\pi \in \Pi} V_\pi(s),
\end{align}
for a given initial state $s \in \cS$.\footnote{
The existence of $P\opt$ for \eqref{def_robust_value} and $\pi\opt$ for \eqref{def_rmdp_general} will be discussed in Appendix~\ref{sec_appendix}.
Consequently we will replace $\inf$ (resp.\ $\sup$) by $\min$ (resp.\ $\max$) going forward.}
Robust MDP problem~\eqref{def_rmdp_general} allows the controller to acknowledge potentially imprecise knowledge or possible deviation of the transition dynamics in MDPs, and instead seeks a robust policy that would perform reasonably well for any kernel $P \in \cP$.
The ambiguity set is typically constructed to include the set of plausible transition dynamics that conform closely with historical data or expert knowledge. 
For instance, it can be constructed from pre-collected real-world trajectories \cite{li2025towards, wiesemann2013robust}, or by varying key simulation parameters of digital simulators \cite{todorov2012mujoco, coumans2016pybullet, hubbs2020or}.
Of course, \eqref{def_rmdp_general} reduces to the {classical} non-robust MDP  when $\cP$ is a singleton.

\revm{Central to solving~\eqref{def_rmdp_general} is the robust policy evaluation subproblem~\eqref{def_robust_value}, for which the prevailing approach is to solve the following dynamic programming equation
\begin{equation}\label{robust-bellman-evaluation}
v_\pi(s)=\inf_{P\in\cP}\sum_{a\in\cA} \pi(a|s) \Bigg(r(s,a) + \gamma \int_{\cS} v_\pi(s')\diff P(s'|s,a)\Bigg).
\end{equation}
\yan{
The operator defined by the right hand side of above \eqref{robust-bellman-evaluation} is often referred to as the robust Bellman operator. Under standard continuity conditions on $r$ and $P \in \cP$, the robust Bellman operator is a contraction and hence has a unique fixed point that can be computed by fixed-point (value) iterations. 
}
It is worth noting that for a general ambiguity set $\cP$, the fixed point $v_\pi$ of~\eqref{robust-bellman-evaluation} does not coincide with the robust value function~$V_\pi$.
These two functions coincide under additional structural assumptions on $\cP$ commonly referred to as rectangularity~\cite{nilim2005robust, iyengar2005robust, wiesemann2013robust, le2007robust, goyal2018robust, li2023rectangularity}. In particular, $\mathrm{s}$-rectangularity assumes that ambiguity set factorizes across states:
\begin{equation}
\label{def_k_mixture_ambiguity}
   \mac P = \{ P\in\cP_0 \mid P(\cdot | s, \cdot) \in\mac P_s ~ \forall s\in\mac S\},
\end{equation}
where $\cP_s = \cbr{P(\cdot|s, \cdot): P \in \cP}$ is referred to as the statewise marginal.
Here $\mathrm{s}$-rectangularity is of particular interest to us, as  any robust MDP  with $v_\pi = V_\pi$ for all $\pi$ can be reformulated into an equivalent one with a $\mathrm{s}$-rectangular ambiguity set \cite{li2023rectangularity}.
Rectangularity certifies time consistency and the existence of optimal stationary Markovian policies, and gives efficient computational schemes for convex $\cP$ based on value iterations.
The importance of {rectangularity}  is further emphasized in \cite{wiesemann2013robust}, 
which shows that evaluating~$V_\pi$  for a fixed policy $\pi$ becomes NP-hard for non-$\mathrm{s}$-rectangular ambiguity set even for finite state spaces \revm{and when~$\cP$ is a polyhedron}.}

\yan{We briefly review the existing solution methods for solving \eqref{def_rmdp_general}.}
\revm{When~$\cP$, or a close approximation thereof, is described by algebraic equations whose parameters can be stored efficiently in memory}, then \eqref{def_rmdp_general} can be addressed with methods from robust dynamic programming,
including (approximate) policy and value iterations \cite{ho2021partial, pmlr-v151-panaganti22a, pmlr-v206-xu23h, ramesh2024distributionally, shi2023curious, nilim2005robust, iyengar2005robust, wiesemann2013robust,goyal2018robust, pmlr-v235-kumar24b}.
We refer to such solution techniques as model-based methods. 
On the other hand, when $\cP$ \revm{cannot be represented explicitly within the available memory budget,} one must resort to a model-free approach.
\revm{Such methods do not require storing the transition kernel and instead approximate expectations using samples.}
This, for instance, includes Q-learning-based method using dual representation of distributionally robust counterpart of Bellman operators \cite{pmlr-v206-wang23b, pmlr-v162-liu22a, panaganti2022robust}, 
or policy gradient methods \cite{wang2022policy, wang2023policy, li2022first, zhou2023natural, kumar2023policy} equipped with robust temporal difference learning for robust policy evaluation.
Complementary to the above lines of development,  solution methods for \eqref{def_rmdp_general} without relying on dynamic programming equations have also been recently studied in \cite{li2023policy, lin2024single, kumar2025dual}.

Despite the recent progress on robust MDPs, it should be noted that the aforementioned methods come with 
finite-time global convergence guarantees only for finite state spaces.
This can be largely attributed to two main reasons.
First, the ambiguity set $\cP$ becomes infinite dimensional for continuous state space.
\yan{In particular, as each state is associated with its transition probability distribution upon taking an action,  
saving~$\cP$ requires saving the transition distributions associated with uncountably many states, thus rendering the direct application of model-based methods impractical.} 
Second, existing model-free methods such as $Q$-learning~\citep{pmlr-v206-wang23b, pmlr-v162-liu22a, panaganti2022robust} or policy gradient \cite{wang2022policy, wang2023policy, li2022first, zhou2023natural, kumar2023policy} often require evaluation of the robust Bellman operator for every state, which is infeasible for continuous state spaces.
\yan{
A common way to handle continuous state spaces in non-robust MDPs is to approximate the value function with a parameterized function class, and perform the approximate (robust) Bellman update within this restricted space.}
However, such commonly used approximate computation schemes, namely, regression-based methods and fixed-point iterations in a parameterized function space, can easily break down for robust MDPs \cite{li2023first,tamar2014scaling}. 
In particular, regression-based methods minimize the robust Bellman residual of \eqref{robust-bellman-evaluation}, and hence need to consider a non-convex objective due to nonlinearity of the robust Bellman operator \cite[Section~3.2]{tamar2014scaling}.
On the other hand, fixed-point iteration methods for continuous state spaces lose the contraction property associated with the Bellman operator in the robust setting, \yan{
as the fixed-point iteration therein is defined by the composition of the robust Bellman operator (contractive in $\cL_\infty$-norm) and the   projection operator  onto the function space in  $\cL^2$-norm.
}\footnote{Unless there is no approximation error for the value function. This is often referred to as the Bellman completeness condition (see, \emph{e.g.}, \cite{JMLR:v9:munos08a}).}
{
\yan{Consequently methods based on fixed-point iterations, in particular, temporal-difference learning that is widely used for non-robust MDPs, can exhibit divergence behavior even with linear approximation, unless with restrictive assumptions on $\cP$ and the discount factor.}
In particular,~\cite{tamar2014scaling} requires the existence of a nominal kernel $P_{\mathrm{nom}} \in \cP$ such that $\sup_{s \in \cS, a \in \cA, P \in \cP} \norm{P(\cdot|s,a)/P_{\mathrm{nom}}(\cdot|s,a)}_\infty < 1/\gamma$, and a divergence example of fixed-point iterations with linear approximations is constructed therein in the absence of this condition. 
 For a given $\cP$, the condition implicitly requires a small enough discount factor, and it remains unclear how this condition can be efficiently verified. 
Indeed, \cite{zhou2023natural} establishes that this condition can even fail to hold for  $\phi$-divergence-based ambiguity sets regardless of the discount factor. Subsequently, \cite{zhou2023natural} proposes two alternative classes of structured ambiguity sets, but both require $\cP$ to have a small diameter of $\cO(1-\gamma)$ to ensure the contraction property of the corresponding fixed-point operator.}
To the best of our knowledge, 
there seems to be no existing method that can achieve global convergence for robust MDPs on continuous state spaces, without {making restrictive assumptions on the discount factor or the diameter of~$\cP$.}

In this paper, we show that by exploiting structural properties of $\cP$, one can design globally convergent computational schemes for solving \eqref{def_rmdp_general}, even if the state space $\cS$ is continuous.
\revm{
The considered structural properties build upon the observation that  
for any ambiguity set $\cP$, whenever the solution $v_\pi$ of dynamic programming equation \eqref{robust-bellman-evaluation} coincides with $V_\pi$ defined in~\eqref{def_rmdp_general}, then there must exist a counterpart of~\eqref{def_rmdp_general} that has the same value function and optimal policy,  with an ambiguity set that is the  convex and $\mathrm{s}$-rectangular hull of $\cP$~\citep{li2023rectangularity}.
\yan{Consequently any solution method for robust MDPs that utilizes dynamic programming equations \eqref{robust-bellman-evaluation} can be viewed as equivalently solving \eqref{def_rmdp_general} with a $\mathrm{s}$-rectangular ambiguity set, where the statewise marginal $\cP_s$ is the convex hull of its extreme points.
}
Motivated by this observation, we focus our attention to ambiguity sets that are $\mathrm{s}$-rectangular and further instantiate each $\cP_s$ as a convex mixture of $K$ generating kernels $P_k(\cdot|s,\cdot)$, \emph{i.e.,} $\cP_s=\{\sum_{k\in\cK} \rw_k  P_k(\cdot | s,\cdot) \mid \rw  \in W \subseteq \Delta_{[K]}\}$. 

Our mixture ambiguity sets also admit a dynamic game interpretation.
Specifically, we consider a two-player dynamic game in which the controller adopts a randomized Borel measurable policy $\pi:\cS\to\Delta_{\cA}$, while nature adopts a randomized Borel measurable policy $\omega:\cS\to W\subseteq\Delta_{[K]}$.
Let us denote by $\Omega$ the set of all feasible policies $\omega$ of the nature.
At any state $s \in \cS$, the controller chooses $a \sim \pi(\cdot|s)$ and nature chooses $k \sim \omega(\cdot|s)$ simultaneously and without observing each other's action.
The next state~$s'$ follows the distribution given by 
$P_k(\cdot | s,a),$
and the immediate reward (resp.\ cost) for the controller (resp.\ nature) is given by $r(s,a)$.
For any given nature's policy $\omega \in \Omega$, consider 
$P^\omega \in \cP$ defined as 
\begin{align*}
    P^\omega(\cdot|s, \cdot) = \sum_{k\in\cK} \omega(k|s) P_k(\cdot|s, \cdot)\quad \forall s \in \cS.
\end{align*}
The dynamic game formulation then seeks a solution to the associated max–min problem
\begin{align}\label{def_dynamic_game}
    \max_{\pi \in \Pi} \min_{\omega \in \Omega} V^\omega_\pi (s) .
\end{align}
It is also shown in~\citep{li2023rectangularity} that  \eqref{def_rmdp_general} with all existing rectangular ambiguity sets is indeed equivalent to the dynamic form~\eqref{def_dynamic_game}. 

The proposed mixture ambiguity set naturally arises from a variety of physics-based (stochastic) differential equations, and digital simulators used in supply chain management, portfolio optimization, and robotics \cite{li2021reinforcement, todorov2012mujoco, coumans2016pybullet, hubbs2020or, da2014stochastic}. The mixture ambiguity set naturally extends group distributionally robust optimization~\citep{hu2018does} to dynamic decision-making settings, while also providing a convenient model for the stochastic optimal control problems.
We will discuss these relevant examples in detail in Section~\ref{sec_ambiguity_set}.
}

For the considered ambiguity set $\cP$, we propose a model-free method that operates in continuous state spaces while using only $\tilde{\cO}(1/\epsilon^2)$ samples to solve both robust policy evaluation~\eqref{def_robust_value} and policy optimization \eqref{def_rmdp_general} to~$\epsilon$ accuracy.
\revm{Since~\eqref{def_rmdp_general} reduces to a non-robust MDP when all generating kernels coincide, the information-theoretic lower bound $\Omega(1/\epsilon^2)$ for non-robust MDPs~\citep[Theorem~5.5]{sidford2018near} applies, implying that our $\tilde{O}(1/\epsilon^2)$ dependence is optimal up to logarithmic factors.}
Notably, the proposed method only requires sample access to generating kernels,  without requiring knowing or saving~$\cP$ a priori, 
and can be implemented with a computational/memory budget that does not depend on the size of the state space.

\vspace{0.05in} 
Our contributions can be summarized as follows.

First,  for robust policy evaluation, we propose a stochastic first-order algorithm that can estimate the robust value function \eqref{def_robust_value} up to $\epsilon$ accuracy  with   $\cO(1/\epsilon^2)$ number of samples drawn from the generating kernels. 
In contrast to existing approaches for continuous-state robust MDPs, the proposed method does not require any additional artificial assumptions on the diameter of $\cP$ or the discount factor \cite{tamar2014scaling, zhou2023natural}, and converges globally in the presence of linear approximation, up to the function approximation error. 
Notably the accuracy certificate is stated in a high probability sense. 
\textit{En route}, we also establish the high probability convergence of temporal difference learning over unbounded domains.
\yan{
The techniques we develop here (see Proposition \ref{prop-bounded-TD}) can  be straightforwardly generalized to establish boundedness of gradient-based stochastic approximation methods over unbounded domain, which could be of independent interest. 
}

Second, for robust policy optimization, we adopt the dynamic viewpoint of formulation \eqref{def_rmdp_general} (see, \emph{e.g.}, \cite{li2023rectangularity}), and present an approximate version of policy iteration method for zero-sum Markov game over continuous state spaces.
Notably the presented policy iteration method takes the previously discussed robust policy evaluation method as a subroutine, and can be implemented efficiently for continuous state spaces.
In particular, the high probability accuracy certificate of robust policy evaluation appears to be essential for the global convergence for a class of policy iteration-based methods.
We further establish an $\tilde{\cO}(1/\epsilon^2)$ sample complexity for the approximate policy iteration method for finding an $\epsilon$-optimal policy of \eqref{def_rmdp_general}.

The obtained sample complexities of $\tilde{O}(1/\epsilon^2)$ for policy evaluation and optimization are order-wise optimal \cite{sidford2018near}. 
To the best of our knowledge, this seems to be the first computationally feasible method for solving continuous state space robust MDPs, with optimal global performance guarantees. 
The proposed framework also directly applies to infinite-horizon zero-sum Markov game, 
\revm{and to our knowledge, our sample complexity results of $\tilde{\cO}(1/\epsilon^2)$ strictly improve on the best known guarantees 
in Markov games with continuous state space~\cite{zhao2022provably}.}

The rest of the paper is organized as follows. 
{Section~\ref{sec_ambiguity_set} introduces the mixture ambiguity set considered in this manuscript and establishes its key structural properties. Section~\ref{sec:RPE} presents an algorithmic framework for robust policy evaluation and analyzes its iteration and sample complexities. Section~\ref{sec:RPO} then presents a method for robust policy optimization. Finally, we conclude in Section~\ref{sec:outro} and  discuss possible future~directions.}

\paragraph{Notation.} 
For any Borel measurable set $X\subseteq\R^d$, we use $\Delta_X$ to denote the set of probability measures defined over~$X$,
and
use $\cL^2(X, \nu)$ to denote the space of all Borel functions $f:X\to\mathbb R$ with $\mathbb E_{X\sim\nu}[ f(X)^2]<\infty$.
When set $X$ is finite, we use $\Delta_X$ to denote the probability simplex defined over $X$.
\revm{We denote the Kullback-Leibler divergence by
$\mathsf{D}(\rw',\rw)= \sum_{k \in \cK} \rw'_k \log \rbr{\rw'_k/\rw_k} $ for all $
\rw,\rw'\in\Delta_{\cK}.$}
For any $n > 0$,   $e_i$ denotes the $i$-th standard basis vector in $\RR^n$.
Finally, we denote $[k] = \cbr{1, \ldots, k}$ for any $k > 0$.

%% file: robust_mdp.tex

\section{Robust MDP with Mixture Ambiguity Set}\label{sec_ambiguity_set}

Throughout the rest of our discussions, we focus on the following structured ambiguity set $\cP$, which we term the mixture ambiguity set. 

\begin{definition}[Mixture Ambiguity Set]\label{def_mix_amb_set}
\revm{Consider a given finite set of transition kernels $\{P_k\}_{k\in\cK}$ where $\cK=[K]$ and $P_k(\cdot|s,a)$ is continuous in the weak topology with respect to $s$ for any $a\in\cA$. In addition, consider as well a given closed set $W \subseteq \Delta_{\cK}$.}
The mixture ambiguity set $\cP$  is the $\mathrm{s}$-rectangular set where $\cP_s$ is defined as
\begin{equation}~\label{def:simplex-Ps}
    \mac P_{s} = \left\{\sum_{k\in\cK} \rw_k  P_k(\cdot | s,\cdot) \mid \rw  \in W 
    \right\}.
\end{equation}
We refer to $\cbr{P_k}_{k\in\cK}$ as the generating kernels of $\cP$.
\end{definition}

Definition \ref{def_mix_amb_set} states that the transition probability of any kernel $ P \in \cP$ from any state $s$ falls within the convex hull of the transition probabilities corresponding to generating kernels.
\revm{If $W=\Delta_{\cK}$, then the ambiguity set is the smallest $\mathrm s$-rectangular and convex set that contains the $K$ generating kernels, where taking the convex hull and the $\mathrm s$-rectangular hull are interchangeable operations. Allowing \(W\) to be a strict subset of \(\Delta_{\cK}\) provides additional modeling flexibility. {In addition, Lemma~\ref{lemma-exist-opt-policy} shows that} the optimal policy~$\pi^\star$ of \eqref{def_rmdp_general} exists and does not depend on the initial state $s \in \cS$.
The ambiguity set~\eqref{def_mix_amb_set} can be viewed as a special case of the more general construction
\begin{equation*}
    \mac P_{s} = \left\{ \left(\sum_{k\in\cK} \rw_{k,a}  P_k(\cdot | s,a) \right)_{a\in\cA} \mid \rw  \in W \right\}
\end{equation*}
where the mixture weights are allowed to depend on the action~$a$. If $W \subseteq \Delta_{\cK}$, with the same weight vector shared across all actions, then this construction reduces to~\eqref{def:simplex-Ps}. On the other hand, if $W = (\Delta_{\cK})^A$, the generalized set coincides with the smallest $(s,a)$-rectangular convex ambiguity set containing the $K$ generating kernels. Intermediate choices $W \subseteq (\Delta_{\cK})^A$ provide additional modeling flexibility by allowing partial coupling of the mixture weights across actions. For clarity of exposition, we restrict attention to the simpler formulation above, although extending our results to this more general class of ambiguity sets is straightforward.}

Mixture ambiguity sets can be convenient for modeling continuous-state MDPs, where the transition kernel, despite being infinite dimensional, corresponds to physical systems that are determined by finite-dimensional parameters.
Notably, this subsumes
physics-based stochastic (partial) differential equations that are typically governed by a finite set of parameters with physical meaning (\emph{e.g.}, diffusion and damping coefficients \cite{da2014stochastic}).
{For instance, the motion of a robotic arm is often described by a second-order ordinary differential equation derived from Newtonian mechanics, where the unknown parameters include joint friction, actuator gains, and link masses that determine how applied torques translate into motion~\citep{todorov2012mujoco,siciliano2009robotics}.
These continuous-time differential equations are typically discretized (for instance, via Euler or Runge–Kutta schemes) before being modeled as a discrete-time MDP.} 
It also includes a variety of digital simulators used in robotics \cite{todorov2012mujoco, coumans2016pybullet},  portfolio optimization \cite{hubbs2020or, boyd2017multi}, and supply chain management \cite{hubbs2020or}.
For instance, in multi-echelon inventory control \cite{hubbs2020or}, the key simulation parameters governing the transition dynamics only involve the Poisson rate of exogenous random demand.

\revm{

We discuss here some potential constructions of $\cP$ for various applications. 
To facilitate our discussion, let us denote parameter space $\Theta \subseteq \RR^d$, and let $P_\theta$ be the transition kernel corresponding to parameter $\theta \in \Theta$.
\begin{example}[{Multi-echelon inventory control~\citep{hubbs2020or}}]\label{inventory-example}
The state $s=(x,y)\in\cS\subseteq\RR^{2E}$ consists of 
on-hand inventories (\emph{i.e.,} stock available) $x=(x_1,\ldots,x_E)$ and pipeline inventories (\emph{i.e.,} outstanding orders) $y=(y_1,\ldots,y_E)$, and the action
$a=(a_1,\ldots,a_E)\in\cA$ represents the ordering quantities. 
Exogenous customer demand $d\in\mathbb R_+^E$ is generated from a distribution $\PP_\theta$ parameterized by $\theta\in\Theta\subseteq\RR$, \emph{e.g.,} $\theta$ may represent the arrival rate of a single Poisson process.
Given $(s,a,d)$, the next state is given by $s' = F(s,a,d),$ where $F=(F_x,F_y)$ is defined through
\begin{align*}
(F_x(s,a,d))_i=x_i+ y_{i-1}- y_i- d_i, \quad (F_y(s,a,d))_i&=a_i \quad \forall i=1,\dots,E,
\end{align*}
with the convention $y_0 = 0$. This induces a transition kernel
\[P_\theta(\cB | s,a)=\PP_\theta(F(s,a,d)\in \cB)\]
for any Borel set $\cB\subseteq\cS.$
The one-period reward function $r(s,a)$ encodes holding, backlog, and ordering costs.
Because the parameter space $\Theta$ is low dimensional, one may consider discretizing the parameter space into $\hat{\Theta} = \cbr{\theta_k}_{k\in\cK}$, that is, consider a finite set of plausible demand intensities (\emph{e.g.}, low, normal, and peak seasons), and correspondingly let 
 $P_k = P_{\theta_k}$ and $W = \Delta_{\cK}$.
This robust MDP problem can be in some sense viewed as seeking the policy that is robust against every possible transition kernel within $\cbr{P_\theta: \theta \in \Theta}$.
\end{example}

\begin{example}
[Drone navigation under dynamic wind conditions~\citep{sydney2013dynamic}]\label{drone-example}
Consider a drone navigating in three-dimensional space.
The state $s = (p, v) \in \cS \subseteq \RR^{6}$ consists of the drone's position $p \in \RR^3$ and velocity $v \in \RR^3$, and the action $a \in \cA$ represents a high-level control input that specifies the desired translational acceleration. The per-step reward $r(s,a)$ encodes trajectory-tracking accuracy and energy consumption. Given the current state $s=(p,v)$ and action $a$, the next state $s'=(p',v')$ evolves according to
\begin{equation*}
    p' = p + v\,\Delta t, \quad
    v' = v + \bigl(a/m + g + w(\xi)\bigr)\Delta t + \sigma\zeta,
\end{equation*}
where $m$ is the drone mass, $g$ is the gravitational acceleration vector, \(w:\RR^d\to\RR^3\) maps an environmental parameter \(\xi\in\RR^d\) to a wind-induced acceleration disturbance, $\sigma > 0$ is a noise scale, and~$\zeta$ is an independent standard Gaussian noise vector. The parameter $\xi$ encodes characteristics such as wind speed, direction, turbulence intensity, and gust frequency and is sampled from a time-varying distribution $\rho_t$, reflecting the fact that wind conditions shift unpredictably during flight.
Conditioning on a realization~$\xi$, the transition from $(s,a)$ to $s'$ is governed by the Gaussian noise~$\zeta$, yielding a generating kernel $P_\xi(\cdot | s,a)$. Taking the expectation over $\xi \sim \rho_t$ produces the effective transition $\EE_{\xi \sim \rho_t}\bigl[P_\xi(\cdot | s, a)\bigr]$,
which is a mixture of the generating kernels $\{P_\xi\}_{\xi \in \RR^d}$.
In drone navigation, a nominal distribution \(\rho_{\mathrm{nom}}\) for the wind parameter \(\xi\) is often available from historical weather records, forecast models, or prior knowledge of the flight region. Yet actual wind conditions may evolve significantly during flight because of local turbulence, gusts, or other environmental effects, so that the true distribution \(\rho_t\) deviates from \(\rho_{\mathrm{nom}}\). To account for this uncertainty, we model nature as adversarially selecting \(\rho_t\) from a divergence ball around the nominal distribution:
\[
    \rho_t \in \cD
    = \bigl\{ \rho \in \Delta_{\RR^d} : \mathsf{D}_\phi(\rho, \rho_{\mathrm{nom}}) \leq \tau \bigr\},
\]
where $\mathsf{D}_\phi$ is a $\phi$-divergence and $\tau > 0$ controls the degree of robustness.
This induces a mixture ambiguity set
\begin{equation}\label{eq:soc-ambiguity-set}
    \cP_s^0
    = \Bigl\{
        \EE_{\xi \sim \rho}\bigl[P_\xi(\cdot \mid s, \cdot)\bigr]
        : \rho \in \cD
    \Bigr\}.
\end{equation}
\end{example}
}

\revm{ Example \ref{drone-example} above in fact is not specific to drone navigation. We next formalize it in a general stochastic optimal control framework and show that it naturally induces a mixture ambiguity set over transition kernels.
\begin{example}[Stochastic optimal control]\label{soc-example}
The stochastic optimal control model considers state transition according to
\begin{equation}\label{eq:soc-dynamics}
    s_{t+1} = f(s_t, a_t, \xi_t, \zeta_t),
\end{equation}
where $s_t \in \cS \subseteq \RR^S$ denotes the state,
$a_t \in \cA$ denotes the control action,
$\xi_t \in \RR^d$ denotes random physical parameter following potentially time-varying distribution $\rho_t$,
and $\zeta_t$ denotes the purely exogenous noise following distribution $\nu$.
For instance, in Example \ref{drone-example}, $\xi_t$ denotes the corresponding wind speed, while $\zeta_t$ corresponds to the additive Gaussian noise. 
Conditioning on a realization $\xi_t = \xi$,
the transition from $(s_t, a_t)$ to $s_{t+1}$
can be described by a kernel
\begin{equation}\label{eq:soc-generating-kernel}
    P_\xi(\cB \mid s, a)
    = \PP_{\zeta \sim \nu}\bigl(f(s, a, \xi, \zeta) \in \cB\bigr)
\end{equation}
for any Borel set $\cB\subseteq\cS.$
Taking the expectation over $\xi_t \sim \rho_t$,
the effective one-step transition kernel becomes the averaged kernel $\EE_{\xi \sim \rho_t}[P_\xi(\cdot \mid s, a)]$. Hence, to hedge against potential distribution shift of physical data, one can consider selecting $\rho_t$ from an ambiguity set  $\cD \subseteq \Delta_{\RR^d} $, and the corresponding family of admissible transition kernels takes the form
\begin{align*}
   \cP_s^0= \Bigl\{
        \EE_{\xi \sim \rho}\bigl[P_\xi(\cdot \mid s,a)\bigr]
        : \rho \in \cD
    \Bigr\},
\end{align*}
which is a mixture ambiguity set generated by the family \(\{P_\xi\}_{\xi \in \RR^d}\).
\end{example}
\begin{remark}
If the generating kernel family $\{P_\xi\}_{\xi \in \RR^d}$ is
indexed by a potentially continuous $\xi$,
the mixture ambiguity set \eqref{eq:soc-ambiguity-set} differs from Definition \ref{def_mix_amb_set} in that the latter assumes a finite number of generating kernels.
Nevertheless we can proceed as follows. 
Let $\xi^{(1)}, \ldots, \xi^{(K)}$ be i.i.d.\ samples following distribution $\rho_{\mathrm{nom}}$,
and consider
\begin{equation}\label{eq:soc-sampled-ambiguity}
    \cP_s
    = \biggl\{
        \sum_{k \in \cK} w_k \, P_{\xi^{(k)}}(\cdot \mid s, \cdot)
        : \mathsf{D}_\phi(w, \hat{\rho}_K) \leq \tau,
        \ w \in \Delta_\cK
    \biggr\},
\end{equation}
where $\hat{\rho}_K$ is the empirical distribution over $\{\xi^{(k)}\}_{k \in \cK}$.
Under Lipschitz continuity of the generating kernels $P_\xi(\cdot|s, \cdot)$ with respect to the state and by sample average approximation results for distributionally robust stochastic programming with $\phi$-divergences~\cite{li2026sample,shapiro2017distributionally},
solving the robust MDP~\eqref{def_rmdp_general} with
$\cP = \prod_{s \in \cS} \cP_s$
approximates the true problem~\eqref{def_rmdp_general} with ambiguity set
$\cP^0 = \prod_{s \in \cS} \cP_s^0$
up to an error of order $\cO(\mathrm{poly}(d)/\sqrt{K})$.
In particular, the required number of generating kernels $K$
depends polynomially on~$d$ and is \emph{independent} of the state space.
\end{remark}
The mixture ambiguity set~\eqref{def:simplex-Ps} offers several modeling advantages for robust MDPs.
A first advantage is its interpretability: the size and shape of the ambiguity set are directly determined by the dispersion among the generating kernels.
When the generating kernels are similar to one another, the resulting ambiguity set is tight and the worst-case policy is only mildly conservative; when they are more distinct, the ambiguity set expands accordingly and the robust policy becomes more conservative.
This interpretation is natural in applications such as multi-echelon inventory control~\citep{hubbs2020or} (\emph{cf.} Example~\ref{inventory-example}), where uncertainty can be represented by a small number of demand regimes, such as different demand intensity levels or seasonal patterns.
Each regime~$k$ specifies a demand model for the exogenous demand process $d$, which in turn induces a transition kernel $P_k(\cdot\mid s,a)$.
If these demand regimes generate similar kernels, then the decision maker faces limited model uncertainty; if they generate markedly different kernels, then the ambiguity set appropriately reflects a higher level of uncertainty.
\yan{This contrasts with commonly used approaches for constructing ambiguity sets in robust MDPs, which consider a probability ball centered  at some nominal kernel $\hat P$, with a radius that depends on whether uncertainty arises from estimation error or from deployment-time model misspecification.
}
If~$\hat P$ is estimated from transaction logs and deployment matches the logging environment, the radius can be calibrated to statistical error -- for example, of order $\tilde O(\sqrt{\log(1/\delta)/N})$ to obtain out-of-sample guarantees from $N$ observed transitions at confidence level $1-\delta$~\citep{wiesemann2013robust}.
If deployment involves regime shifts (\emph{e.g.}, season changes, promotions, supplier disruptions, or sim-to-real mismatch~\citep{tobin2017domain,peng2018sim}), then the radius must represent structural misspecification rather than sampling noise, and it is not identified by $N$ alone.
For instance, even if one knows an upper bound on the difference in the Poisson rate of exogenous demand between simulation and true demand distribution, it is unclear how to translate this knowledge into an appropriate radius of the corresponding Kullback-Leibler divergence ball centered at the data-generating kernel. This is because the same absolute difference in Poisson parameter can result in an orders-of-magnitude difference in Kullback-Leibler divergence.  

Second, it is commonly acknowledged that a major modeling benefit of $\mathrm{s}$-rectangular sets is that they allow one to capture dependence among the action-conditional transition kernels $\{P(\cdot\mid s,a)\}_{a\in\cA}$ through a proper construction of the marginal ambiguity set $\cP_s$, defined in \eqref{def:simplex-Ps}.
However, while such flexibility is desirable, it is often unclear how this cross-action coupling should be specified in a principled manner unless substantial prior knowledge of the underlying dynamics is available.
In particular, one should take caution not to impose artificially hard-coded coupling across actions when such dependence is not directly justified by the data-generating mechanism.
In applications such as multi-echelon inventory control, the coupling among $\{P(\cdot\mid s,a)\}_{a\in\cA}$ should be induced by the common latent demand mechanism rather than introduced manually.
A prior approach toward obtaining such coupling is to construct the marginal ambiguity set $\cP_s$ using statistical procedures, for example through super-level sets of the maximum likelihood estimator~\citep{wiesemann2013robust}.
Unfortunately, this approach requires the transition kernel to be sufficiently simple so that the analytical form of $P(\cdot\mid s,a)$ is available.
Recent approaches instead enforce coupling through aggregated constraints such as $\sum_{a\in\cA} \mathsf D(P(\cdot\mid s,a),\hat P(\cdot\mid s,a))\le \tau$~\citep{li2025near}, yet the induced cross-action coupling depends on the particular choice of aggregation and is not directly implied by the underlying demand mechanism.
In contrast, the mixture ambiguity set couples actions through the generating kernels $\{P_k\}_{k\in\cK}$, so the coupling is inherited from the modeled demand regimes rather than manually hard-coded, and it remains applicable even when the kernels have no analytical form available.}

We proceed to introduce some structural properties of nature's cost-minimizing MDP, which will prove useful in our ensuing algorithmic development. 
These properties arise naturally from the structure of the proposed mixture ambiguity set.

\begin{definition}
\label{def_nature_state_action_func}
    For any joint policy $(\pi, \omega) \in \Pi \times \Omega$, define the joint action-value function through
    \begin{align*}
        Q^\omega_\pi(s, a, k) = r(s,a) +  \gamma \int_{\cS} V^\omega_\pi(s') P_k(\diff s' | s, a) 
        \quad \forall (s,a,k) \in \cS \times \cA \times \cK.
    \end{align*}
Nature's action-value function associated with its policy $\omega\in\Omega$ is then defined by
\begin{align*}
    {H^\omega_\pi(s, k) = 
        \sum_{a \in \cA} \pi(a|s) Q^\omega_\pi(s,a,k)
    \quad \forall (s,k) \in \cS \times \cK.}
\end{align*}
\end{definition}

Let $\PP_\pi^\omega (\cdot|S_0=s_0)$ denote the probability law of the Markov process $\cbr{(S_t, A_t)}_{t \geq 0}$ induced by $\pi$ and $P^\omega$ given an initial state $S_0 = s_0$ (Ionescu Tulcea Theorem~\cite{tulcea1949mesures}).

\begin{definition}\label{def_discounted_visitation}
We define the discounted state visitation measure induced by joint policy $(\pi, \omega) \in  \Pi \times \Omega$ through 
\begin{align*}
d^\omega_{\pi}(\cB|s_0)=(1-\gamma) \sum_{t=0}^{\infty} \gamma^t \mb P^\omega_\pi\left(S_t \in \cB| S_0=s_0\right), 
\end{align*}
for any Borel set $\cB \subseteq \cS$.
\end{definition}
The following lemma then characterizes the difference of nature's value functions for any pair of its policies. 
\revm{It reduces the global value gap to the discounted expectation of local one-step Bellman update differences under the comparison occupancy measure.
}

\begin{lemma}\label{lemma:pdl}

For any nature's policies $\omega,\omega'\in\Omega$, and any $s\in\mac S$, we have
\begin{align}\label{eq_perf_dff_Q}
V^{\omega'}_\pi(s)-V^{\omega}_\pi(s)&=\frac{1}{1-\gamma} \int_{\cS} \!\left[
\sum_{k\in\cK} (\omega'(k|s') - \omega(k|s') )H^\omega_\pi(s',k)
\right]d^{\omega'}_\pi(\diff s'|s).
\end{align}
In addition, let 
\begin{align}\label{eq_def_psi}
\psi_\pi^\omega(s,\rw)= \gamma \int { V^\omega_\pi(s') (P^{\rw }_\pi- P^{\omega(s) }_\pi)( \diff s'|s) } \quad \forall \rw \in W.
\end{align}
We then have
\begin{align}\label{def_perf_diff_psi}
V^{\omega'}_\pi(s)-V^{\omega}_\pi(s) =\frac{1}{1-\gamma} \int_{\cS} \psi^{\omega}_\pi (s', \omega'(\cdot|s'))d^{\omega'}_\pi(\diff s'|s).
\end{align}
\end{lemma}

\begin{proof}
The proof of \eqref{eq_perf_dff_Q} follows directly from the performance difference lemma~\citep{kakade2002approximately} applied to the nature's cost minimizing MDP and Definition \ref{def_nature_state_action_func}.
In addition, \eqref{def_perf_diff_psi} follows from \eqref{eq_perf_dff_Q}, together with the definitions of $\psi^\omega_\pi$ in~\eqref{eq_def_psi} and $H^\omega_\pi$ in Definition~\ref{def_nature_state_action_func}.
\end{proof}

In the following sections, we proceed to introduce efficient solution methods for the robust policy evaluation~\eqref{def_robust_value} and optimization \eqref{def_rmdp_general} problems. 
Notably the proposed methods only require sampling access to the generating kernels $\cbr{P_k}_{k\in\cK}$.
This can be particularly convenient for implementation purposes, as simulating a random process  can be considerably easier than constructing or storing its distribution explicitly.

%% file: robust_eval_outer.tex


\section{Robust Policy Evaluation}\label{sec:RPE}

In this section, we turn our attention to the robust policy evaluation problem \eqref{def_robust_value}.
Given any to-be-evaluated policy $\pi$, as discussed in Section~\ref{sec_ambiguity_set}, 
it suffices to compute the optimal value function of nature's cost-minimizing MDP. 
With this in mind, we present in Algorithm~\ref{alg:SPDA} the details of the 
proposed stochastic method for robust policy evaluation.

\begin{algorithm}[b!]
\caption{Stochastic dual averaging for robust policy evaluation}
\label{alg:SPDA}
\begin{algorithmic}[1]
  \REQUIRE Controller's policy $\pi\in\Pi$, initial nature's policy $\omega^{(0)}\in\Omega$, number of iterations $M$, stepsizes $\cbr{\alpha_m}_{m =0}^M$ and regularization parameters $\cbr{\lambda_m}_{m=0}^M$.
  \FOR{$m = 0,1,2,\dots, M$}
      \STATE \label{line:td-oracle}
      Compute $\hat Q_{m} \approx Q^{\omega^{(m)}}_\pi$ 
      \revise{
      with $(\omega^{(m)}, \pi)$ and sample access to generating kernels $\cbr{P_k}_{k \in \cK}$ 
      } 
      \STATE Set 
\begin{align}\label{eq_estimate_V_H}
      \hat{V}_m (\cdot) = \sum_{a \in \cA} \sum_{k \in \cK} \omega^{(m)}(k|\cdot) \pi(a|\cdot)  \hat{Q}_m(\cdot,a, k), ~~ \text{and} ~~
          \hat{H}_m(\cdot, k) =  \sum_{a \in \cA} \pi(a|\cdot)  \hat{Q}_m(\cdot,a, k)
      \end{align}
    \STATE 
    Update $\omega^{(m+1)}(\cdot|s)$ through~\eqref{line:ftrl}
  \ENDFOR
  \RETURN
  $
     {\hat V_{\pi}}=\sum_{m=1}^{M} {\vartheta_m}\,\hat V_{{m}}$ and 
     $
     {\hat Q_{\pi}} = \sum_{m=1}^{M} {\vartheta_m}\,\hat Q_{{m}}$, where 
  $\vartheta_m=\alpha_m/(\sum_{j=1}^M \alpha_j)$ for $m=1,2,\ldots,M$
\end{algorithmic}
\end{algorithm}

\revise{
At any given iteration $m$, Algorithm \ref{alg:SPDA} operates by first estimating the nature's action-value function $H^{\omega^{(m)}}_\pi$. 
The estimation follows a two-step procedure, 
defined by first estimating the joint action-value function $Q^{\omega^{(m)}}_\pi$ given the joint policy $(\pi, \omega^{(m)})$, 
followed by estimating $H^{\omega^{(m)}}_\pi$ in view of Definition~\ref{def_nature_state_action_func}.
We will introduce a subroutine (Algorithm \ref{alg:TD}) in Section \ref{ssec:TD}  that provides an efficient estimation $\hat Q_m$ of $Q^\omega_\pi$ for any given $(\pi, \omega) \in \Pi \times \Omega$, by using samples drawn from the generating kernels $\cbr{P_k}_{k\in\cK}$.
To facilitate our ensuing discussion, we denote by~$\cD_m$ as the samples from generating kernels at iteration~$m$, and let $\cF_{m-1}$ denote the $\sigma$-algebra generated by $(\cD_0,\ldots,\cD_{m-1})$.
With the estimation of $H^{\omega^{(m)}}_\pi$ in place, Algorithm \ref{alg:SPDA} then updates the nature's policy using the dual-averaging-type update \cite{nesterov2009primal,ju2022policy},
\begin{align}\label{line:ftrl}
    \omega^{(m+1)}(\cdot|s) \gets \argmin_{\rw\in W}   \sum_{i=0}^{m}\alpha_{i} 
    \sum_{k\in\cK} \rw_k \hat{H}_{i}(s,k) 
+ \lambda_{m}h(\rw) ,
\end{align}
where the distance-generating function $h(\rw) = \sum_{k \in \cK} \rw_k \log(\rw_k)$ corresponds to the negative entropy function,
and $\hat{H}_i$ denotes the estimation of $H^{\omega^{(i)}}_\pi$, where $\alpha_m\ge 0$ are the stepsizes and $\lambda_m\ge0 $ are regularization parameters.

It is worth noting that from an implementation perspective, the policy update in~\eqref{line:ftrl} is an implicit, query-based representation of~$\omega^{(m+1)}$, rather than a full update over the entire state space.  In particular, when the algorithm needs to access~$\omega^{(m+1)}$ at a state~$s$, it suffices to compute the local distribution~$\omega^{(m+1)}(\cdot | s)$ on demand. There is therefore no need to explicitly compute or store~$\omega^{(m+1)}(\cdot | s)$ for every state~$s \in \cS$.  As a result, the state-space cardinality does not appear explicitly in the per-iteration computational complexity of Algorithm~\ref{alg:SPDA}. This will be further discussed within Section \ref{ssec:TD}.
}

\revise{
\begin{remark}
It should be noted here that a similar update has also been studied in~\cite{ju2022policy} for solving non-robust MDPs. 
Hence it is worth discussing some important differences between the development herein and \cite{ju2022policy}, detailed~below.
The method in \cite{ju2022policy} considers solving non-robust MDPs and focuses on controlling the bias of estimating the optimal value function.
In our context of robust policy evaluation \eqref{def_robust_value} and dynamic game introduced in Section \ref{sec_ambiguity_set}, this could only control the expected difference between the estimated robust value function $\hat{V}_\pi$ and the true robust value function $V_\pi$, \emph{i.e.},
$ \|\EE [\hat{V}_\pi] - V_\pi\|_\infty$.
On the other hand, for the purpose of policy evaluation, one would ideally seek an estimation accuracy in a high probability sense instead of only controlling the bias.
Indeed,  we will proceed to establish an $\tilde{\cO}(1/\varepsilon^2)$ sample complexity for obtaining
$ \|\hat{V}_\pi - V_\pi\|_\infty \leq \varepsilon$ with high probability.
Note the convergence guarantee on the bias $ \|\EE [\hat{V}_\pi] - V_\pi\|_\infty$ is even weaker than on the mean-squared error typically performed for non-robust policy evaluation \cite{kotsalis2022simple, bhandari2018finite},
which in turn is weaker than a high probability bound.
It turns out that this seemingly simple goal of boosting the accuracy certificate from bias to high probability creates technical challenges that will become evident as we proceed.
In addition, such a conversion does not seem to be amenable to existing techniques originally developed from stochastic convex optimization with a strongly convex objective \cite{davis2021low, nemirovskij1983problem,hsu2016loss}.
Perhaps most importantly, we will demonstrate in Section \ref{sec:RPO} that the high probability control on robust policy evaluation appears to be essential for robust policy optimization (Remark \ref{remark_high_prob}).
Indeed, it remains unclear to us whether controlling the bias for $\hat{V}_\pi$ would be sufficient for any approximate policy iteration-based framework applied to the robust policy optimization~\eqref{def_rmdp_general}. 
\end{remark}
}

Our subsequent discussion on the convergence of Algorithm \ref{alg:SPDA} requires the following condition 
on the stochastic error of 
$\hat{Q}_m$.
\begin{subequations}
\begin{condition}\label{assu-approx-H-V}
\revise{For $\delta \in (0,1)$, the estimator $\hat{Q}_m$ in Algorithm \ref{alg:SPDA} satisfies }  
    \begin{align}\label{assump_w}
        \|\EE_\pi^{\omega^{(m)}}[ \hat Q_m | \mac F_{m-1}] - Q^{\omega^{(m)}}_\pi \|_\infty \le \underline{\varepsilon}, ~~~ \| \hat Q_m \|_\infty \,\le B, \quad \EE_\pi^{\omega^{(m)}}\big[ \| \hat Q_m - Q^{\omega^{(m)}}_\pi\|_\infty^2 | \mac F_{m-1}\big] \le J
    \end{align}
for some $B, J, \underline{\varepsilon} > 0$
    with probability $1-\delta/(4M)$ for every $m=1,\ldots,M$, where $\{\omega^{(m)}\}_{m=1}^M$ are the iterates generated by Algorithm~\ref{alg:SPDA}.
\end{condition}

\revise{
It should be noted that the ensuing convergence characterization for Algorithm \ref{alg:SPDA} holds for arbitrary estimation accuracy $(B, J, \underline{\varepsilon})$ as defined in Condition \ref{assu-approx-H-V}.
Of course, the quality of the obtained $\hat{V}_\pi$ by Algorithm~\ref{alg:SPDA} would in turn depend on $(B, J, \underline{\varepsilon})$.
For a given stochastic subroutine that estimates $\hat{Q}_m$, one can also in principle determine its associated estimation accuracy.  
We will proceed with this approach in Section \ref{ssec:TD}, by introducing a concrete stochastic estimation subroutine (Algorithm \ref{alg:TD}) to build $\hat Q_m$ and subsequently determine its $(B, J, \underline{\varepsilon})$.
This will in turn help us determine the total sample complexity of Algorithm~\ref{alg:SPDA} in Section \ref{ssec:putting-it-together}.
It turns out that in order to establish high probability convergence,  both certifying Condition~\ref{assu-approx-H-V} with non-vacuous $(B, J, \underline{\varepsilon})$ and utilizing it within the analysis of Algorithm \ref{alg:SPDA} are nontrivial. 
\textit{En route}, we develop some new probabilistic tools for the analysis of Algorithm~\ref{alg:SPDA} and \ref{alg:TD}, 
which could be of independent interest for  general stochastic approximation methods. 
In comparison, \cite{ju2022policy} directly assumes the access to oracles for estimating $\hat{Q}_m$ and focuses on the weaker notion of convergence for the value function.  
}

From Definition \ref{def_nature_state_action_func} and Condition~\ref{assu-approx-H-V}, together with the observation that $\omega^{(m)}$ is $\cF_{m-1}$-measurable, we have for every $m=1,\ldots,M$ with probability $1-\delta/(4M)$ that 
\begin{align}
   &  \|\EE_\pi^{\omega^{(m)}}[ \hat V_m | \mac F_{m-1}] - V^{\omega^{(m)}}_\pi \|_\infty \le \underline{\varepsilon}, ~~~ \| \hat V_m \|_\infty \,\le B,  \quad \EE_\pi^{\omega^{(m)}}\big[ \| \hat V_m - V^{\omega^{(m)}}_\pi\|_\infty^2 | \mac F_{m-1}\big] \le J, \label{ineq_norm_bias_V}
   \end{align}
and
   \begin{align}
   &  \|\EE_\pi^{\omega^{(m)}}[ \hat H_m | \mac F_{m-1}] - H^{\omega^{(m)}}_\pi \|_\infty \le \underline{\varepsilon}, ~~~ \| \hat H_m \|_\infty \,\le B, \quad \EE_\pi^{\omega^{(m)}}\big[ \| \hat H_m - H^{\omega^{(m)}}_\pi\|_\infty^2 | \mac F_{m-1}\big] \le J. \label{ineq_norm_bias_H}
\end{align}
\end{subequations}
To facilitate our discussion, let us
denote 
\begin{align*}
    & \psi_m (s, \rw) = \sum_{k\in\cK} (\rw_k- \omega^{(m)}(k|s)) H^{\omega^{(m)}}_\pi(s,k), 
    \\
    & \hat{\psi}_m(s,\rw) = \sum_{k\in\cK} (\rw_k- \omega^{(m)}(k|s))  \hat{H}_m(s,k), \\
    & \hat{\Psi}_m(s,\rw)  = \sum_{i=0}^m \alpha_{i} \hat{\psi}_{i} (s, \rw). 
\end{align*}
\revise{One can readily verify that $\psi_m(s,\rw)$ represents the difference in nature's true action-value function at state~$s$ between its current policy at iteration $m$ and any alternative state-wise policy $\rw \in W$, while $\hat{\psi}_m(s,\rw)$ denotes its sample approximation. As will be clarified later,  $\hat{\Psi}_m(s,\rw)$ can be understood as a proxy for the cumulative progress of Algorithm \ref{alg:SPDA}.
}
Note that from Definition~\ref{def_nature_state_action_func} and~\eqref{eq_def_psi}, it is clear that 
$\psi_m = \psi^{\omega^{(m)}}_\pi$.
In addition, let us define the stochastic error
\begin{align*}
    \delta_m  = \hat{\psi}_m  - \psi_m  .
\end{align*}
We now discuss an immediate consequence of Condition \ref{assu-approx-H-V}. 
In particular, in the ensuing Lemma \ref{lemma:weighted-value-errors}, we show that the accumulated stochastic errors in Algorithm~\ref{alg:SPDA} with $M$ total iterations grow at the order of 
$\cO(\sqrt{M})$ with high probability.
Its key ingredient, Lemma \ref{lemma-fixed-prop34} (Appendix \ref{sec_appendix}),  generalizes the Azuma-Hoeffding inequality in the sense that it does not require boundedness of stochastic errors almost surely.
\revise{
It appears to us that this technique might also be applied in the general context of stochastic approximation algorithms with potentially unbounded stochastic updates. }

\begin{lemma}
\label{lemma:weighted-value-errors}
Suppose that Condition~\ref{assu-approx-H-V} holds, and set $\alpha_m=\sqrt{m}$ for $m\le M$. In addition, let $\bar B=B+\frac{1}{1-\gamma}$ and denote by $\omega^\star_\pi$ the optimal policy for the inner minimization problem in \eqref{def_dynamic_game}. For any $\varepsilon\ge \underline{\varepsilon}$, the following hold with probability at least $1-\delta/2$ for all $(s,a,k)\in\cS\times\cA\times\cK$.
\begin{enumerate}[label = (\roman*)]
\item \label{item-vartheta-Vhat}
$\displaystyle\left|\sum_{m=1}^M \vartheta_m(\hat V_m(s)-V_\pi^{\omega^{(m)}}(s))\right| 
\leq \varepsilon + \frac{3 \bar B}{\sqrt{M}} \sqrt{\log \left(\frac{8 M}{\delta}\right)} + \sqrt{\frac{\revise{J}\delta}{M}}   $;
\item \label{item-vartheta-Hhat}
$\displaystyle\sum_{m=1}^M \!\frac{\vartheta_m}{1-\gamma} \int_{\cS}\delta_m(s^{\prime}, \omega^\star_\pi(\cdot|s^{\prime}))d_{\pi}^{\,\omega^\star_\pi}(\diff s'|s)\le \frac{2}{1-\gamma}\left( \varepsilon + \frac{3 \bar B}{\sqrt{M}}\sqrt{\log\left(\frac{8 M}{\delta}\right)}+ \sqrt{\frac{\revise{J}\delta}{M}}\right)  $;
\item \label{item-vartheta-Qhat}
$\displaystyle\left|\sum_{m=1}^M \vartheta_m(\hat Q_m(s,a,k)-Q_\pi^{\omega^{(m)}}(s,a,k))\right| 
\leq \varepsilon + \frac{3 \bar B}{\sqrt{M}} \sqrt{\log \left(\frac{8 M}{\delta}\right)} + \sqrt{\frac{\revise{J}\delta}{M}}  $.
\end{enumerate}
\end{lemma}
\begin{proof}
Observe first that the choice of $\alpha_m$ gives
$\sum_{m=1}^M \alpha_m= \sum_{m=1}^M \sqrt m \;\ge\; \int_{0}^{M} \!\!\sqrt x\,dx
=\frac{2}{3}M^{3/2}$. Hence, we have
\begin{equation}\label{ineq:vartheta-bound}
    \sum_{m=1}^M \vartheta_m^2=\frac{\sum_{m=1}^M \alpha_m^2}{(\sum_{j=1}^M \alpha_j)^2}\le\frac{M(M+1)/2}{(\frac{2}{3}M^{3/2})^2}\le\frac{9}{4M}.
\end{equation}
For Assertion~\ref{item-vartheta-Vhat}, since \eqref{ineq_norm_bias_V} holds with probability $1-\delta / (4M)$, we have for every fixed $s\in\cS$ that
\begin{align*}
\left|\sum_{m=1}^M \vartheta_m(\hat V_m(s)-V_\pi^{\omega^{(m)}}(s))\right| 
\leq \varepsilon + \bar B \sqrt{2 \sum_{m=1}^M \vartheta_m^2 \log \left(\frac{8 M}{\delta}\right)} + \sqrt{\frac{\revise{J}\delta}{M}} \leq \varepsilon + \frac{3 \bar B}{\sqrt{M}} \sqrt{\log \left(\frac{8 M}{\delta}\right)} + \sqrt{\frac{\revise{J}\delta}{M}}
\end{align*}
holds with probability at least $1-\delta / 2$. The first inequality in the above expression follows from Lemma~\ref{lemma-fixed-prop34}, and the second inequality holds because of~\eqref{ineq:vartheta-bound}.
This observation concludes the proof for Assertion~\ref{item-vartheta-Vhat}.
Assertion~\ref{item-vartheta-Hhat} (resp.\ Assertion \ref{item-vartheta-Qhat}) follows from a similar argument where~\eqref{ineq_norm_bias_H} (resp.\ Condition~\ref{assu-approx-H-V}) is invoked instead of~\eqref{ineq_norm_bias_V}.
\revise{Note that the additional factor of $2$ on the right-hand side of Assertion~\ref{item-vartheta-Hhat} arises from the observation that
\[\int_{\cS}\delta_m(s^{\prime}, \omega^\star_\pi(\cdot|s^{\prime}))d_{\pi}^{\,\omega^\star_\pi}(\diff s'|s)= \sum_{k\in\cK} \omega^\star_\pi(k|s) (\hat{H}_m(s,k) -  H^{\omega^{(m)}}_\pi(s,k)) + \sum_{k\in\cK}\omega^{(m)}(k|s)( H^{\omega^{(m)}}_\pi(s, k) - \hat{H}_m(s, k)) .\] 
Thus, the claim follows.}
\end{proof}

We proceed by presenting the following basic property for each update \eqref{line:ftrl} of Algorithm~\ref{alg:SPDA} \revise{that allows us to later analyze the overall progress in Lemma~\ref{lemma:pda-general}.} 
Note that the Bregman divergence induced by the distance-generating function $h$ corresponds to the Kullback-Leibler divergence $\mathsf{D}$, that is, we have $\mathsf{D}(\rw',\rw)=h(\rw') - h(\rw) - \inner{\nabla h(\rw)}{\rw' - \rw} $ for all $\rw,\rw'\in\Delta_{\cK}$.

\begin{lemma}
\label{lemma:three-point}
Define $\hat{\Psi}_{-1}=0$ and $\lambda_{-1}=0$,
and suppose $\lambda_m \ge \lambda_{m-1}$ for every $m\ge 1$. Then, for all $m\ge 0 $ and $s\in\mac S$, the iterates $\{\omega^{(m)}\}_{m=1}^M$ generated by Algorithm~\ref{alg:SPDA} satisfy
\begin{enumerate}[label = (\roman*)]
    \item \label{item-3point} $\hat{\Psi}_m(s, \omega^{(m+1)}(\cdot|s))+\lambda_m \mathsf{D}(
\rw,\omega^{(m+1)}(\cdot|s)) + \lambda_m h(\omega^{(m+1)}(\cdot|s)) - \lambda_m h(\rw) \leq \hat{\Psi}_m(s, \rw)$ for every $\rw\in W$,
    \item \label{eq-descent-pda}
    $\alpha_m \hat{\psi}_m(s, \omega^{(m+1)}(\cdot|s)) \leq \hat{\Psi}_m(s, \omega^{(m+1)}(\cdot|s))-\hat{\Psi}_{m-1}(s, \omega^{(m)}(\cdot|s))-\lambda_{m-1} \mathsf{D}(\omega^{(m+1)}(\cdot|s), \omega^{(m)}(\cdot|s))-\lambda_{m-1} h(\omega^{(m)}(\cdot|s)) + \lambda_{m-1} h( \omega^{(m+1)}(\cdot|s)).$
\end{enumerate}
\end{lemma}

\begin{proof}
    For Assertion~\ref{item-3point}, note that the first-order optimality condition of \eqref{line:ftrl} in Algorithm~\ref{alg:SPDA} implies
    \begin{align*}
    \sum_{i=0}^{m}\alpha_{i}\! \sum_{k\in\cK}\rw_k \hat{H}_{i}(s,k)  - 
        \sum_{i=0}^{m}\alpha_{i}\! \sum_{k\in\cK} \omega^{(m+1)}(k|s)\hat{H}_{i}(s,k) + \lambda_{m}  \langle \nabla h(\omega^{(m+1)}(\cdot|s)), \rw-\omega^{(m+1)}(\cdot|s)\rangle \ge 0 \revise{\;\; \forall \rw\in W.}
    \end{align*}
The claim then follows from the definition of Bregman divergence $\mathsf{D}(\cdot, \cdot)$ and the definition of $\hat\Psi_m$.
    For Assertion~\ref{eq-descent-pda}, we have
    \begin{equation*}
        \begin{split}
            \alpha_m \hat{\psi}_m(s, \omega^{(m+1)}(\cdot|s)) 
            & =\hat{\Psi}_m(s, \omega^{(m+1)}(\cdot|s))-\hat{\Psi}_{m-1}(s, \omega^{(m+1)}(\cdot|s)) 
         \\&\leq \hat{\Psi}_m(s, \omega^{(m+1)}(\cdot|s))-\hat{\Psi}_{m-1}(s, \omega^{(m)}(\cdot|s))-\lambda_{m-1} \mathsf{D}(\omega^{(m+1)}(\cdot|s),\omega^{(m)}(\cdot|s))
         \\&\quad -\lambda_{m-1} h(\omega^{(m)}(\cdot|s)) + \lambda_{m-1} h( \omega^{(m+1)}(\cdot|s)),
        \end{split}
    \end{equation*}
where the equality holds by the construction of $\hat\Psi_m$, and the inequality \revise{uses Assertion~\ref{item-3point} with $m\leftarrow m-1$ and $\rw \leftarrow \omega^{(m+1)}(\cdot|s)$}. 
\end{proof}

Building on Lemmas~\ref{lemma:weighted-value-errors} and~\ref{lemma:three-point}, we next establish \revise{a general high-probability characterization of Algorithm~\ref{alg:SPDA} with $\alpha_m = \sqrt{m}$ and any specification of parameters $\{\lambda_m\}_{m\in[M]}$.} 
This will in turn be used later to show that the output $\hat{V}_\pi$ of Algorithm~\ref{alg:SPDA} approaches the robust value function~$V_\pi$ with a high probability.

\begin{lemma}
\label{lemma:pda-general}
Suppose that Condition~\ref{assu-approx-H-V} holds\revise{, and define $\hat{\Psi}_{-1}=0$ and $\lambda_{-1}=0$. In addition, set $\alpha_m=\sqrt{m}$ for $m\le M$}. Then, for any $\varepsilon \ge \underline{\varepsilon}$, with probability at least $1-\delta/2$, we have
\begin{equation*}
\begin{split}
0 \le   \sum_{m=1}^M \vartheta_m V_\pi^{\omega^{(m)}}(s) - V_\pi(s)
&\leq  \frac{2\varepsilon}{1-\gamma}+\left(\sum_{m=1}^M \sqrt{m}\right)^{\!\!\!-1}\!\!\left(\sum_{m=1}^M \frac{m B^2}{2(1-\gamma) \lambda_{m-1}} + \frac{2 \lambda_M \log(|\cK|)}{1-\gamma}\!\right) 
\\&\qquad\qquad\qquad\qquad\qquad+ \frac{6  \bar B}{\sqrt{M}(1-\gamma)} \sqrt{\log\! \left(\frac{8 M}{\delta}\right)} + \frac{2}{1-\gamma}\sqrt{\frac{\revise{J}\delta}{M}}  \qquad \forall s\in\mac S.
\end{split}
\end{equation*} 
\end{lemma}
\begin{proof}
\revise{The lower bound follows trivially from the observation that $V_\pi^{\omega^{(m)}}(s) - V_\pi(s)\ge 0$ for all $m=0,\ldots,M$.
For the upper bound, we first sum} Lemma~\ref{lemma:three-point}\ref{eq-descent-pda} from $m=0$ to $M$. \revise{Observing that $\hat\Psi_{-1}=0$, $\lambda_{-1}=0$ and $h(\rw)\le 0$ for all $\rw\in W$}, we then obtain 
\begin{align*}
\alpha_0 \hat{\psi}_0(s, \omega^{(1)}(\cdot|s)) 
&\leq  \hat{\Psi}_M(s, \omega^{(M+1)}(\cdot|s))-\sum_{m=1}^M \lambda_{m-1} \mathsf{D}(\omega^{(m+1)}(\cdot|s),\omega^{(m)}(\cdot|s))-\sum_{m=1}^M \alpha_m \hat{\psi}_m (s,\omega^{(m+1)}(\cdot|s)) \\
&\quad - \lambda_{M-1} h( \omega^{(M)}(\cdot|s)) \\
& \le \hat{\Psi}_M(s, \omega^{(M+1)}(\cdot|s))-\frac{1}{2}\sum_{m=1}^M \lambda_{m-1}\|\omega^{(m+1)}(\cdot|s)\!-\!\omega^{(m)}(\cdot|s)\|_1^2 
\\&\quad -  \!\sum_{m=1}^M \!\alpha_m  \sum_{k\in\cK} (\omega^{(m+1)}(k|s)-\omega^{(m)}(k|s))\hat H_m(s, k) - \lambda_{M-1} h( \omega^{(M)}(\cdot|s))\\
&\le \hat{\Psi}_M(s, \omega^{(M+1)}(\cdot|s))-\frac{1}{2}\sum_{m=1}^M \lambda_{m-1}\|\omega^{(m+1)}(\cdot|s)\!-\!\omega^{(m)}(\cdot|s)\|_1^2  
\\&\quad-\sum_{m=1}^M \alpha_m \cbr{\max_{k\in\cK}\|\hat H_m(s,k)\|_2}\|\omega^{(m+1)}(\cdot|s)-\omega^{(m)}(\cdot|s)\|_1  - \lambda_{M-1} h( \omega^{(M)}(\cdot|s))\\
&\le \hat{\Psi}_M(s, \omega^{(M+1)}(\cdot|s))+\sum_{m=1}^M \frac{\alpha_m^2  \cbr{\max_{k\in\cK}\|\hat H_m(s,k)\|_2}^2}{2 \lambda_{m-1}}- \lambda_{M-1} h( \omega^{(M)}(\cdot|s))
\\& \leq \hat{\Psi}_M(s, \omega^\star_\pi(s))+\!\sum_{m=1}^M \frac{\alpha_m^2 \cbr{\max_{k\in\cK}\|\hat H_m(s,k)\|_2}^2}{2 \lambda_{m-1}}  + 2\lambda_{M} \log(|\cK|),
\end{align*}
where the second inequality holds by Pinsker's inequality \revise{(Lemma~\ref{lemma:pinsker})} and the definition of $\hat{\psi}$, the third inequality uses Hölder's inequality, and the fourth inequality follows \revise{from the simple fact that  
$ bx - a x^2  \leq  b^2 /(4 a)$ for any $a>0$ and $x, b \in \RR$.
}
Finally, the last inequality follows from Lemma~\ref{lemma:three-point}\ref{item-3point} with \revise{$\rw=\omega^\star_\pi(\cdot|s)$ and the observations that $\mathsf{D}(
\rw,\omega^{(m+1)}(\cdot|s)) - h(\rw) \ge 0$ and $|h(\rw)|\le \log(|\cK|) $ for all $\rw\in W$}.
Since $\alpha_0=0$, integrating both sides 
of the above inequality with respect to $d_{\pi}^{\,\omega^\star_\pi}(\cdot|s)$ and using the definition of $\hat\Psi_M$ and~\eqref{def_perf_diff_psi} gives
\begin{align*}
   0&\leq \sum_{m=1}^M \alpha_m \int_{\cS} \psi_m(s,  \omega^\star_\pi(\cdot|s))d_{\pi}^{\,\omega^\star_\pi}(\diff s'|s)
   \\&\quad+\sum_{m=1}^M \frac{\alpha_m^2 \cbr{\max_{k\in\cK}\|\hat H_m(s,k)\|_2}^2}{2 \lambda_{m-1}} +\sum_{m=1}^M \alpha_m \int_{\cS} \delta_m(s,  \omega^\star_\pi(\cdot|s))d_{\pi}^{\,\omega^\star_\pi}(\diff s'|s) +  2 \lambda_{M} \log(|\cK|)
   \\&=\sum_{m=1}^M \alpha_m(1-\gamma)(V_\pi^{\omega^\star_\pi}(s)-V_\pi^{\omega^{(m)}}(s))
\\&\quad+\sum_{m=1}^M \frac{\alpha_m^2 \cbr{\max_{k\in\cK}\|\hat H_m(s,k)\|_2}^2}{2 \lambda_{m-1}}+\sum_{m=1}^M \alpha_m \int_{\cS} \delta_m(s^{\prime}, \omega^\star_\pi(\cdot|s^{\prime})) d_{\pi}^{\,\omega^\star_\pi}(\diff s'|s)+  2 \lambda_{M} \log(|\cK|)
\\&\revise{=\left(\sum_{j=1}^{M}\alpha_{j}\right)\sum_{m=1}^M \vartheta_m(1-\gamma)(V_\pi(s)-V_\pi^{\omega^{(m)}}(s))}
\\&\quad \revise{+\left(\sum_{j=1}^{M}\alpha_{j}\right)^2\sum_{m=1}^M \frac{\vartheta_m^2 B^2}{2 \lambda_{m-1}} + 2  \left(\sum_{j=1}^{M}\alpha_{j}\right)\sum_{m=1}^M \vartheta_m \left( \varepsilon + \frac{3 \bar B}{\sqrt{M}}\sqrt{\log\left(\frac{8 M}{\delta}\right)} + \sqrt{\frac{\revise{J}\delta}{M}}\right) + 2 \lambda_{M} \log(|\cK|),}
\end{align*}
where the first equality follows from Lemma~\ref{lemma:pdl}, \revise{and the second equality uses the observations that $V^{\omega^\star_\pi}_\pi = V_\pi$, 
$\vartheta_m=\alpha_m/(\sum_{j=1}^{M}\alpha_{j})$, together with Lemma~\ref{lemma:weighted-value-errors}\ref{item-vartheta-Hhat}.}
\revise{The upper bound then follows by the choice of $\alpha_m=\sqrt{m}$. This observation concludes the proof.}
\end{proof}

With Lemma \ref{lemma:weighted-value-errors} and \ref{lemma:pda-general} in place,
we are now ready to specify the concrete  choice of 
$\cbr{\lambda_m}_{m \in [M]}$, and correspondingly establish the high probability convergence of $\hat V_\pi$ and $\hat Q_\pi$ towards
both the robust value function~$V_\pi$ and the robust joint action-value function $Q_\pi$ defined through
\begin{align}\label{eq_def_opt_joint_q_give_pi}
        Q_\pi(s,a,k) = \min_{\omega \in \Omega} Q^\omega_\pi(s,a,k) \quad \forall (s,a,k) \in \cS \times \cA \times \cK.
    \end{align}

\begin{theorem}
\label{thm-V-spda}
Suppose that Condition~\ref{assu-approx-H-V} holds.  
Set  
$$
\alpha_m=\sqrt{m}, \quad \lambda_m=\frac{(m+1) B}{2 \sqrt{\log(|\cK|)}}, \quad \forall m \leq M
$$
in Algorithm~\ref{alg:SPDA}.
Then, for any $\varepsilon\ge 4\underline{\varepsilon}/(1-\gamma)$, 
the total number of iterations required by Algorithm~\ref{alg:SPDA} to output
\begin{align*}
   -\varepsilon\le \hat V_\pi(s)-V_\pi(s) \le \varepsilon,\quad 
   -\varepsilon\le \hat Q_\pi(s,a,k)-Q_\pi(s,a,k) \le \varepsilon \quad \forall (s,a,k)\in\cS\times\cA\times\cK
\end{align*}
with probability at least $1-\delta$ is bounded by
\begin{align}\label{iter_comp_eval}
M=\cO\left(
\frac{\bar B^{2}\log(|\cK|)}{(1-\gamma)^{2} \varepsilon^2}\log\rbr{\frac{\bar B^2}{\delta(1-\gamma)^2\varepsilon^2}}+\frac{J\delta }{\varepsilon^2}\right).
\end{align}
\end{theorem}

\begin{proof}
Adding $\sum_{m=1}^M \vartheta_m (\hat V_m(s)-V_\pi^{\omega^{(m)}}(s))$ to both sides in the stated inequality of Lemma~\ref{lemma:pda-general} and observing that $(1-\gamma)\varepsilon/4\ge \underline{\varepsilon}$, we obtain, with probability at least $1-\delta/2$,
\begin{equation}\label{pda-general-hat-V}
\begin{split}
\sum_{m=1}^M \vartheta_m {(\hat V_m(s)-V_\pi^{\omega^{(m)}}(s))}\le \sum_{m=1}^M \vartheta_m { \hat V_m(s)-V_\pi(s) }\leq \sum_{m=1}^M \vartheta_m {( \hat V_m(s)-V_\pi^{\omega^{(m)}}(s))}+ \frac{\varepsilon}{2}+C(M,\delta),
\end{split}
\end{equation}
where 
\begin{equation*}
    C(M,\delta)=\left(\sum_{m=1}^M \sqrt{m}\!\right)^{\!\!\!-1}\!\!\!\left(\sum_{m=1}^M \frac{m B^2}{2(1-\gamma) \lambda_{m-1}}+\frac{2 \lambda_M \log(|\cK|)}{1-\gamma}\!\right) + \frac{6 \bar B}{\sqrt{M}(1-\gamma)} \sqrt{\log\! \left(\frac{8 M}{\delta}\right)} +\frac{2}{1-\gamma} \sqrt{\frac{\revise{J}\delta}{M}}.
\end{equation*}
Noting $(1-\gamma)\varepsilon/4\ge \underline{\varepsilon}$, we then apply Lemma~\ref{lemma:weighted-value-errors}\ref{item-vartheta-Vhat} to the above inequality, which yields
\begin{equation}\label{ineq-V-m-dependent}
\begin{split}
-\frac{\varepsilon}{4} - \frac{3 \bar B}{\sqrt{M}} \sqrt{\log \left(\frac{8 M}{\delta}\right)} - \sqrt{\frac{\revise{J}\delta}{M}}
\le   \sum_{m=1}^M \vartheta_m \hat V_m(s)-V_{\pi}(s)
\leq  \frac{3\varepsilon}{4} + C(M,\delta) + \frac{3 \bar B}{\sqrt{M}} \sqrt{\log\! \left(\frac{8 M}{\delta}\right)}+ \sqrt{\frac{\revise{J}\delta}{M}}
\end{split}
\end{equation}
with probability $1-\delta$. 
On the other hand, in view of Definition~\ref{def_nature_state_action_func} and Lemma~\ref{lemma:pda-general}, we have
\begin{equation*}
\begin{split}
0 \le \sum_{m=1}^M \vartheta_m Q_\pi^{\omega^{(m)}}(s,a,k) - Q_\pi(s,a,k)
\le \frac{\varepsilon}{2}+C(M,\delta)
\end{split}
\end{equation*}
with probability $1-\delta/2$.
Adding $\sum_{m=1}^M \vartheta_m (\hat Q_m(s,a,k)-Q_\pi^{\omega^{(m)}}(s,a,k))$ to 
both sides
of the above inequality and applying Lemma~\ref{lemma:weighted-value-errors}\ref{item-vartheta-Qhat} with the observation $(1-\gamma)\varepsilon/4\ge \underline{\varepsilon}$, we obtain with probability $1-\delta$ that
\begin{equation}\label{ineq-Q-m-depedent}
\begin{split}
-\frac{\varepsilon}{4} - \frac{3 \bar B}{\sqrt{M}} \sqrt{\log \left(\frac{8 M}{\delta}\right)}- \sqrt{\frac{\revise{J}\delta}{M}}
&\le \sum_{m=1}^M \vartheta_m \hat Q_m(s,a,k)- Q_\pi(s,a,k)
\\&\leq  \frac{3\varepsilon}{4} + C(M,\delta) + \frac{3 \bar B}{\sqrt{M}} \sqrt{\log\! \left(\frac{8 M}{\delta}\right)}+ \sqrt{\frac{\revise{J}\delta}{M}}.
\end{split}
\end{equation}
Finally, by choosing
$\lambda_m=\frac{(m+1) B }{2 \sqrt{\log(|\cK|)}}$, we have 
$\big(\sum_{m=1}^{M}\sqrt{m}\big)^{-1}\le \big(\int_0^M \sqrt{x}\,\diff x\big)^{-1}=\frac{3}{2} M^{-3/2}$, and $ m/\lambda_{m-1}=2\sqrt{\log(|\cK|)}/B$.
Hence,
choosing
\begin{align*}
M=\cO\left(
\frac{\bar B^{2}\log(|\cK|)}{(1-\gamma)^{2} \varepsilon^2}\log\rbr{\frac{\bar B^2}{\delta(1-\gamma)^2\varepsilon^2}}+\frac{J\delta }{\varepsilon^2}\right)
\end{align*}
ensures that 
\begin{align*}
    C(M,\delta)+\frac{3 \bar B}{\sqrt{M}} \sqrt{\log\! \left(\frac{8 M}{\delta}\right)} + \sqrt{\frac{\revise{J}\delta}{M}} \leq \frac{\varepsilon}{4}.
\end{align*}
Applying the previous observation to~\eqref{ineq-V-m-dependent} and~\eqref{ineq-Q-m-depedent} together with the construction of $\hat V_\pi$ and $\hat Q_\pi$ concludes the proof.
\end{proof}

\revise{
A few remarks are in order before we conclude our discussion in this section. 
First, the requirement on the estimation accuracy $\varepsilon = \Omega(\underline{\varepsilon})$ is natural,
as Algorithm \ref{alg:SPDA} works with an inaccurate estimation of $\hat{Q}_m$ as stated in Condition \ref{assu-approx-H-V}.
}
In addition,  \revise{the first term in the iteration complexity \eqref{iter_comp_eval} obtained in Theorem~\ref{thm-V-spda} scales logarithmically with respect to $1/\delta$, while the second term scales  linearly with respect to $\delta$.} 
As the required confidence level $1-\delta$ approaches $1$, \revise{the resulting complexity is clearly dominated by the first term.}
Consequently, in view of Theorem~\ref{thm-V-spda}, it suffices for Algorithm~\ref{alg:SPDA} to take $\tilde{\cO}(1/\varepsilon^2)$ iterations to output an $\varepsilon$-accurate estimation of the robust value function $V_\pi$ associated with the mixture ambiguity set in Definition~\ref{def_mix_amb_set}. 
Of course, 
this iteration complexity hinges upon Condition \ref{assu-approx-H-V}, which requires accurate estimation of the value function $V^\omega_\pi $ and the action-value function $H^\omega_\pi$ of nature. 
This in turn will be fulfilled by Algorithm \ref{alg:TD}, \revise{which we discuss in detail in Section \ref{ssec:TD} and correspondingly determine its associated values of $(B,J, \underline{\varepsilon})$.}
In particular, we will determine the corresponding sample complexity required to certify Condition~\ref{assu-approx-H-V}.
Finally, we will establish in Section~\ref{ssec:putting-it-together}  the total sample complexity of Algorithm \ref{alg:SPDA}
for estimating the robust value function $V_\pi$.

%% file: td_nature.tex

\subsection{Estimating the Joint Action-value Function} \label{ssec:TD}

We now proceed to introduce a temporal difference (TD) learning-based method \cite{sutton1988learning} for estimating the joint action-value function $Q^\omega_\pi$ associated with any $(\pi, \omega) \in \Pi \times \Omega$.
In particular, we will establish that the proposed method constitutes a valid subroutine in Algorithm \ref{alg:SPDA} that certifies Condition \ref{assu-approx-H-V}.

Consider the Bellman operator
$T_\pi^\omega$ defined through 
\begin{align}\label{eq_bellman_v}
\!\!\!\!(T_\pi^\omega Q)(s, a, k)= r(s,a) + \gamma \int_{\cS}  \sum_{a' \in \cA}\sum_{k' \in \cK} \pi(a'|s')  \omega(k'|s') Q(s', a', k') P_k(\diff s'|s,a)  \;\;\; \forall (s,a,k)\in\mac S\times\cA\times\cK,
\end{align}
for any $Q: \cS \times \cA \times \cK \to \R$. Going forward we also write $\cZ = \cS \times \cA \times \cK$, together with $z \in \cZ$ in short for $z = (s, a, k)$ for some $s \in \cS, a\in \cA, k\in \cK$.
It is clear that $T_\pi^\omega$ maps the space of bounded continuous functions onto itself, and constitutes a contraction operator in $\norm{\cdot}_\infty$ norm. In addition, by dynamic programming equations, it can readily be seen that the unique fixed point of $T^\omega_\pi$ is given by $Q^\omega_\pi$ (Banach fixed point theorem).

\begin{figure}[b] 
\vspace{-0.2in}
\centering
\begin{minipage}[b]{0.53\textwidth}
\vspace{0pt} 
\begin{algorithm}[H]
\caption{Temporal difference learning}\label{alg:TD}
\begin{algorithmic}[1]
\REQUIRE $(\pi, \omega) \in \Pi \times \Omega$, $\eta$, $T$, $\tau$, $S_0 = s \in \cS$
\STATE  $K_0 \sim \omega(\cdot|S_0)$, $A_0 \sim \pi(\cdot|S_0)$, $S_1 \sim P_{K_0}(\cdot|S_0, A_0)$, 
$K_1 \sim \omega(S_1)$, $A_1 \sim \pi(\cdot|S_1)$.
Set 
$\xi_{-1}^\tau = (Z_0, Z_1)$.
\FOR{$t = 0,\ldots,T$}
  \STATE $\xi_{t}^\tau \gets$ \Call{MarkovSampler}{$\xi^\tau_{t-1},\omega, \pi,\tau$}
  \STATE \label{step:TD} $\theta_{t+1} = \theta_t - \eta \hat{F}(\theta_t, \xi_{t}^\tau)$
\ENDFOR
\RETURN { $\hat{Q} (\cdot) =  \phi(\cdot)^\top \theta_{T+1}$ }
\end{algorithmic}
\end{algorithm}
\end{minipage}%
\hfill
\begin{minipage}[b]{0.45\textwidth}
\vspace{0pt}
\makeatletter
\renewcommand{\ALG@name}{Procedure}
\makeatother
\begin{algorithm}[H]
\caption{\textsc{MarkovSampler}($\xi,\omega,\pi,\tau$)}\label{procedure_markov_sampler}
\begin{algorithmic}[1]
\REQUIRE $(\pi, \omega) \in \Pi \times \Omega$, $\tau$, $\xi \in \cZ \times \cZ$
\STATE $\texttt{Z}_0 = \xi(0), \texttt{Z}_1 = \xi(1)$
\FOR{$i=0,\ldots,\tau$}
  \STATE Sample $\texttt{K}_i \sim \omega(\cdot|\texttt{S}_i)$, $\texttt{A}_i \sim \pi(\cdot|\texttt{S}_i)$
  \STATE Sample $\texttt{S}_{i+1}\sim P_{\texttt{K}_i }(\cdot|\texttt{S}_i, \texttt{A}_i)$
  \STATE Set $\texttt{Z}_i = (\texttt{S}_i, \texttt{A}_i, \texttt{K}_i)$  
\ENDFOR
\RETURN $(\texttt{Z}_{\tau-1}, \texttt{Z}_\tau)$ 
\end{algorithmic}
\end{algorithm}
\end{minipage}
\end{figure}

It should be noted that as the set $\cZ$ is continuous, it becomes computationally prohibitive even to represent $Q^\omega_\pi$ with a reasonable computing budget. 
Instead, we consider finding the best possible approximation of $Q^\omega_\pi$ within the linear space spanned by some pre-determined basis functions.
In particular, let $\phi_i: \cZ \to \RR$ denote the $i$-th basis function for $i = 1, \ldots ,d$,
and define the feature mapping $\phi(\cdot) = [\phi_1(\cdot), \ldots, \phi_d(\cdot)]$. 
 Without loss of generality, we assume that $\norm{\phi(\cdot)}_2 \leq 1$.
The constructions of basis functions $\cbr{\phi_i}_{i=1}^d$ and the choice of dimension $d$ have been extensively discussed in the prior literature, for instance, by using random feature mappings that can uniformly approximate a function class with bounded norms in a reproducing kernel Hilbert space \cite{rahimi2007random}. 

We proceed as follows. Let us denote by $\cQ = \cbr{Q_{\theta}(\cdot) \coloneqq \theta^\top \phi(\cdot):  \theta \in \RR^d}$  the linear span of the basis functions. 
One can view $\cQ$ as a proper subspace of $\cL^2(\cZ, \nu^\omega_\pi)$,
where $\nu^\omega_\pi$ corresponds to the stationary measure of Markov process $\cbr{Z_t \coloneqq (S_t,A_t, K_t)}_{t \geq 0}$ with transition law defined by
$A_t \sim \pi(\cdot|S_t), K_t \sim \omega(S_t), S_{t+1} \sim P_{K_t}(\cdot|S_t, A_t)$.
We are interested in finding  $\theta^\omega_\pi \in \RR^d$ such that $Q_{\theta^\omega_\pi}(\cdot) $  provides, in some sense, a reasonable approximation of~$Q^\omega_\pi$.
In particular, from \eqref{eq_bellman_v} and  $T^\omega_\pi Q^\omega_\pi = Q^\omega_\pi$, it is natural to consider finding $\theta^\omega_\pi$ such that 
\begin{align}\label{eq_projected_bellman}
   \Pi_{\cL^2(\cZ, \nu^\omega_\pi)} T^\omega_\pi Q_{\theta^\omega_\pi} = Q_{\theta^\omega_\pi},
\end{align}
where $\Pi_{\cL^2(\cZ, \nu^\omega_\pi)}$ denotes the orthogonal projection onto $\cQ$.\footnote{Clearly, the projection $\Pi_{\cL^2(\cZ, \nu^\omega_\pi)}$ is necessary, otherwise a solution might not exist for \eqref{eq_projected_bellman}.
It should be also noted that projection on $\cZ$ in $\norm{\cdot}_\infty$ norm, although conceptually appealing and naturally leading to value iteration, is in general not implementable as the space $\cZ$ is uncountable.
}
From the optimality condition of the projection $\Pi_{\cL^2(\cZ, \nu^\omega_\pi)}$, it can be readily verified that \eqref{eq_projected_bellman} is equivalent to  
 $F(\theta^\omega_\pi) = 0$, where 
\begin{align}\label{eq_td_op_exact}
F(\theta)= & \int_{\cZ} \phi(z) \bigg[
\phi(z)^\top\theta \;-\; r(s, a)   - \gamma  \int_{ \cS} \sum_{a' \in \cA} \sum_{k' \in \cK } \pi(a'|s') \omega(k'|s') \phi(s', a', k')^\top \theta  P_k(\diff s' | s,a)
\bigg]\nu^\omega_\pi(\diff z) .
\end{align}

We present the details of the proposed TD learning method in Algorithm \ref{alg:TD}.
In a nutshell, the method operates by drawing samples from the generating kernels $\cbr{P_k}_{k\in\cK}$ to construct a stochastic estimation $\hat F$ of the operator $F(\cdot)$ defined in \eqref{eq_td_op_exact}.
Note that the trajectory $\cbr{Z_t}_{t \ge 0}$ is sampled from the generating kernels $\cbr{P_k}_{k \in \cK}$ in a Markovian manner.
In view of the Markovian noise,  we take a similar strategy as in \cite{kotsalis2022simple, yu1994rates} that periodically skips $\tau$ samples 
within the above construction of $\hat{F}$. 
In particular, at the $t$-th iteration of Algorithm~\ref{alg:TD}, the algorithm evaluates the stochastic operator $\hat{F}$ through
\begin{align}\label{eq_hat_F}
\hat{F}(\theta_t, \xi_t^{\tau})=(\langle\phi(Z_t^\tau), \theta_t\rangle-r(S_{t}^\tau,A_t^{\tau})-\gamma\langle\phi((Z_{t}^\tau)' ), \theta_t\rangle) \phi(Z_t^\tau),
\end{align}
where
$\xi_t^\tau = (Z_t^\tau, (Z_t^\tau)')$, generated by Procedure \ref{procedure_markov_sampler},   corresponds to the $t$-th transition pair that is separated from $(t-1)$-th pair by $\tau$ timesteps.  
We will establish later that $\hat{F}$ constructed above serves as an approximately unbiased estimation of $F$ in \eqref{eq_td_op_exact}.

Before we proceed, it is worth mentioning here that we will establish high probability convergence of Algorithm \ref{alg:TD} over an unbounded domain, which should be contrasted with existing convergence analysis in expectation \cite{kotsalis2022simple,bhandari2018finite}.
As will be clarified later, the interdependence of the noise associated with $\hat{F}$ and the norm of the iterates makes the corresponding analysis for high probability results considerably more involved compared to expectation bounds.
In particular, the approach we develop for establishing boundedness of iterates with high probability seems to be new for TD-type methods, and hence may be of independent interest.

For any to-be-evaluated joint policy $(\pi, \omega) \in \Pi \times \Omega$, 
we make the following standard assumption on the Markov process $\cbr{Z_t}_{t\ge0}$
and the feature mapping $\phi$, see~\citep{tsitsiklis1996analysis, kotsalis2022simple}. \revise{The first part of the assumption requires the state process to mix at a geometric rate under $(\pi,\omega)$, whereas the positive-definiteness assumption on $\Sigma$ rules out degenerate feature representations under the stationary distribution.}

\begin{assumption}\label{assu:geometric-ergodicity}
There exist
constants $C>0$ and $\rho\in(0,1)$ such that for every Borel set
$\mac B\subseteq\mac S$,
\begin{align}
\big\lvert
     \PP^\omega_\pi(S_{t+\tau}\in \mac B \mid \mac F_{t-1})-\nu^\omega_\pi(\mac B)
\big\rvert \le C \rho^{\tau}
   \quad\forall\,t\in\mb Z_{++},\ \tau\in\mb Z_+ ,
\end{align}
where $\cF_{t-1}$ denotes the filtration up to iteration $t-1$ of Algorithm~\ref{alg:TD}.
In addition, let  $ \Sigma=\int_{\cZ }\phi(z)\,\phi(z)^\top\nu^\omega_\pi(\diff z)$. 
We assume that $\Sigma \succ 0$.
\end{assumption}

To proceed, let us denote 
\begin{align*}
    \varepsilon_{\mathrm{approx}} =\sup_{\pi \in \Pi, \omega\in\Omega}\|Q_{\theta^\omega_\pi}-Q^\omega_\pi\|_\infty
\end{align*}
as the approximation error of approximating $Q^\omega_\pi$ with linearly parameterized $Q_{\theta^\omega_\pi}$. 
In addition, let 
\begin{align*}
    R=\sup_{\pi \in \Pi, \omega\in\Omega}\|\theta^\omega_\pi\|_2, \ U=2R + 1. 
\end{align*} 
From  \eqref{eq_projected_bellman} and Assumption~\ref{assu:geometric-ergodicity}, it can be readily verified that $\theta^\omega_\pi$ is continuous with respect to $\omega$,
consequently $R < \infty$ since $\Omega$ is compact.

We now proceed to discuss some immediate implications of Assumption \ref{assu:geometric-ergodicity}.
\revise{In particular, Lemma \ref{lemma-property-F} establishes two basic properties of the operator $F(\cdot)$ defined in \eqref{eq_td_op_exact}, namely its Lipschitz continuity and strong monotonicity. These properties will serve as the basis for the analysis of Algorithm~\ref{alg:TD}.}

\begin{lemma}\label{lemma-property-F}
Suppose Assumption~\ref{assu:geometric-ergodicity} holds.
Denote $\mu =\lambda_{\min }(\Sigma)(1-\gamma)$ and $L=\lambda_{\max}(\Sigma)(1+\gamma)$ \revm{where $\lambda_{\min }(\Sigma)$ and $\lambda_{\max }(\Sigma)$ are the smallest and largest eigenvalues of the matrix $\Sigma$, respectively. }
Then, we have
\begin{enumerate}[label = (\roman*)]
    \item \label{item:F-ub}  $\|F(\theta)\|_2\le L\|\theta-\theta^\omega_\pi\|_2$ for any $\theta \in \RR^d$,
    \item \label{item:F:strong:monotone}  $\langle F(\theta), \theta-\theta^\omega_\pi\rangle \ge \mu  \|\theta-\theta^\omega_\pi\|_2^2 $ for any $\theta \in \RR^d$.
\end{enumerate}   
\end{lemma}

To facilitate our discussion, let us define $\mathcal K:\cZ \to \mathbb R^{d\times d}$ through 
\begin{align*}
    \mathcal K(z)=\phi(z) \sbr{ \phi(z)^\top-\gamma \int_{\mathcal S} \sum_{a' \in \cA} \sum_{k' \in \cK} \pi(a'|s') \omega(k'|s') \phi(s',a',k')^\top P_k(\diff s' | s,a) } \quad \forall z\in\cZ.
\end{align*}

\begin{proof}[Proof of Lemma~\ref{lemma-property-F}]
For Assertion~\ref{item:F-ub}, observe first that 
\begin{equation}\label{reexpress-K}
\begin{split}    
\int_{\cZ} \mathcal K(z)\,\nu^\omega_\pi(\diff z)&= \Sigma - \gamma \int_{\cZ\times\cS} \sum_{a' \in \cA} \sum_{k' \in \cK} \pi(a'|s') \omega(k'|s') \phi(z)\phi(z')^\top P_k(\diff s' | s,a) \nu^\omega_\pi(\diff z)
\\&=\Sigma - \gamma \int_{\cZ}\phi(z')\phi(z')^{\!\top}\, \nu^\omega_\pi(\diff z')=(1-\gamma)\Sigma,
\end{split}
\end{equation}
where the first and last equalities follow from the definition of~$\Sigma$, and the second equality holds because $\nu^\omega_\pi$ is the steady-state distribution.
We then have 
\begin{equation}\label{reexpress-F}
    F(\theta)=F(\theta)-F(\theta^\omega_\pi)=\int_{\cZ}
\mathcal K(z)\,\nu^\omega_\pi(\diff z)(\theta-\theta^\omega_\pi)=(1-\gamma)\Sigma(\theta-\theta^\omega_\pi),
\end{equation}
where the first equality holds because $F(\theta^\omega_\pi)=0$, 
the second equality follows from the linearity of $F$ in $\theta$, 
and the third equality uses~\eqref{reexpress-K}.
Assertion~\ref{item:F-ub} then follows by the definition of~$L$.
For Assertion~\ref{item:F:strong:monotone}, we have 
\begin{equation*}
    \begin{split}
        \langle F(\theta), \theta-\theta^\omega_\pi\rangle = 
            (\theta-\theta^\omega_\pi)^\top (1-\gamma)\Sigma (\theta-\theta^\omega_\pi)
\ge \mu  \|\theta-\theta^\omega_\pi\|_2^2,
    \end{split}
\end{equation*}
where the first equality follows from~\eqref{reexpress-F}, and the inequality follows by the definition of~$\mu $.
\end{proof}

Let us define the stochastic error associated with $\hat{F}(\theta_t, \xi_t^\tau)$ in \eqref{eq_hat_F} through 
\begin{align}\label{def_td_err_zeta}
  \zeta_t= \hat{F}(\theta_t, \xi_t^\tau)- F(\theta_t).
\end{align}
Lemma~\ref{lemma:mixing} below establishes that the bias of the estimator $\hat{F}(\theta_t, \xi_t^\tau)$ decays exponentially fast in $\tau$ under Assumption~\ref{assu:geometric-ergodicity}.
\begin{lemma}
\label{lemma:mixing}
Suppose that Assumption~\ref{assu:geometric-ergodicity} holds. We then have
$$
\Big\|F(\theta_t)-\mb E^\omega_\pi\big[\hat{F}(\theta_t, \xi_{t}^{\tau}) \mid \mathcal{F}_{t-1}\big]\Big\|_2 \leq (1+\gamma) C\rho^\tau\|\theta_t-\theta^\omega_\pi\|_2 \quad \PP^\omega_\pi\text{-a.s.}.
$$
\end{lemma}
\begin{proof}
Denote by $\nu_{\tau|t}\in\cP(\cZ)$ the conditional distribution of $Z_{t+\tau}$ given $\cF_{t-1}$, that is, for every Borel set $\mathcal{B}$, $\PP^\omega_\pi[Z_{t+\tau}\in \mac B\mid\mathcal F_{t-1}]=\int_{\mac B} \nu_{\tau|t}(\diff z)$.
We then have
\begin{equation*}
\begin{split}
    \left\|F(\theta_t)-\mb E^\omega_\pi \big[\hat{F}(\theta_t, \xi_{t}^{\tau}) \mid \mathcal{F}_{t-1}\big]\right\|_2
    &= \left\|\int_{\cZ}\!
\cK(z)\,(\nu^\omega_\pi(\diff z)- \nu_{\tau|t}(\diff z)) (\theta_t-\theta^\omega_\pi)\right\|_2
   \\&\le \left\|\int_{\cZ}\cK(z) \,(\nu^\omega_\pi(\diff z)-\nu_{\tau|t}(\diff z)) \right\|_{\mathrm{op}}\|\theta_t-\theta^\omega_\pi\|_2
\\&\le \sup_{z\in\cZ} \|\cK(z)\|_{\mathrm{op}} \, \|\nu^\omega_\pi-\nu_{\tau|t}\|_{\mathrm{TV}} \, \|\theta_t-\theta^\omega_\pi\|_2
\\&\le \sup_{z,z'\in\cZ}\|\phi(z)\|_2(\|\phi(z)\|_2+\gamma \|\phi(z')\|_2) \, \|\nu^\omega_\pi-\nu_{\tau|t}\|_{\mathrm{TV}} \, \|\theta_t-\theta^\omega_\pi\|_2
\\&\le (1+\gamma) 
\|\nu^\omega_\pi-\nu_{\tau|t}\|_{\mathrm{TV}} \, \|\theta_t-\theta^\omega_\pi\|_2
\le(1+\gamma)  C  \rho^{\tau}\|\theta_t-\theta^\omega_\pi\|_2,
\end{split}
\end{equation*}
where the first and the third inequalities follow from Cauchy-Schwarz inequality, the second inequality uses Hölder's inequality, and the fourth inequality holds because $\|\phi(z)\|_2\le 1$ for all $z\in\cZ$. Finally, the last inequality uses Assumption~\ref{assu:geometric-ergodicity}. Thus, the claim follows.
\end{proof}

We are ready to discuss the convergence behavior of Algorithm~\ref{alg:TD}, which will help us specify the concrete values of its parameters $(\eta, T)$ that in turn certify Condition \ref{assu-approx-H-V}.
We begin by first establishing that the bias of the estimate $\theta_t$ exhibits linear convergence.

\begin{lemma}\label{lemma-bias-conv}
Suppose that Assumption~\ref{assu:geometric-ergodicity} holds, set $\tau > (\log(\mu)-\log(C(1+\gamma))) /\log(\rho)$ and define $\tilde\mu  = \mu  - (1+\gamma) C\rho^\tau>0$. If
$
\eta \leq \tilde\mu /L^2,
$
then
$$
\|\mb E^\omega_\pi[\theta_t]-\theta^\omega_\pi\|_2^2 \leq(1-\eta \tilde\mu )^t\|\theta_0-\theta^\omega_\pi\|_2^2 .
$$
\end{lemma}

\begin{proof}
We have
$$
\begin{aligned}
&\|\mb E^\omega_\pi[\theta_{t+1}]-\theta^\omega_\pi\|_2^2 
\\ =  &\|\mb E^\omega_\pi[\theta_t]-\eta F(\mb E^\omega_\pi[\theta_t]) - \eta \mb E^\omega_\pi [\zeta_t] -\theta^\omega_\pi\|_2^2 \\
\leq &  \|\mb E^\omega_\pi[\theta_t]-\theta^\omega_\pi\|_2^2-2 \eta\langle F(\mb E^\omega_\pi[\theta_t]), \mb E^\omega_\pi[\theta_t]-\theta^\omega_\pi\rangle - 2 \eta \langle \mb E^\omega_\pi [\zeta_t], \mb E^\omega_\pi[\theta_t] - \theta^\omega_\pi\rangle +\eta^2\|F(\mb E^\omega_\pi[\theta_t])\|_2^2
\\ \le  & \|\mb E^\omega_\pi[\theta_t]-\theta^\omega_\pi\|_2^2-2 \mu  \eta\|\mb E^\omega_\pi[\theta_t]-\theta^\omega_\pi\|_2^2 + 2 \eta (1+\gamma) C\rho^\tau \|\mb E^\omega_\pi[\theta_t]-\theta^\omega_\pi\|_2^2 +L^2 \eta^2\|\mb E^\omega_\pi[\theta_t]-\theta^\omega_\pi\|_2^2
\\ \le  & (1-2\eta \tilde \mu  + L^2\eta^2) \|\mb E^\omega_\pi[\theta_t]-\theta^\omega_\pi\|_2^2,
\end{aligned}
$$
where the first equality follows by the update rule in Line~\ref{step:TD} of Algorithm~\ref{alg:TD}, and the second inequality holds because of Lemma~\ref{lemma-property-F}\ref{item:F:strong:monotone}, Lemma~\ref{lemma:mixing}, and Lemma~\ref{lemma-property-F}\ref{item:F-ub}. The claim then follows from the choice of~$\eta.$ 
\end{proof}
As will be clarified later in Theorem \ref{thm-TD-V}, Lemma \ref{lemma-bias-conv} will be useful in showing that the output of Algorithm \ref{alg:TD} satisfies the first inequality of Condition \ref{assu-approx-H-V}.
\revise{
To certify the remaining requirements in Condition~\ref{assu-approx-H-V},
we proceed to introduce the following Lemma \ref{lemma:TD-mse} establishing that the distance between the iterate $\theta_t$ generated by Algorithm \ref{alg:TD} and the optimal solution remains bounded in Algorithm \ref{alg:TD} in expectation. 
In particular, we will build upon Lemma~\ref{lemma:TD-mse} and later control the same distance in high probability. 
}

\begin{lemma}\label{lemma:TD-mse}
Suppose that Assumption~\ref{assu:geometric-ergodicity} holds, set $\tau > (\log(\mu)-\log(C(1+\gamma))) /\log(\rho)$ and define  $\tilde\mu  = \mu  - (1+\gamma) C\rho^\tau$. 
If $\eta\le \tilde \mu/8$, then
\begin{align}\label{td_norm_recursive}
  \left\|\theta_{t+1}-\theta^\omega_\pi\right\|_2^2 \leq\left\|\theta_t-\theta^\omega_\pi\right\|_2^2+2 \eta^2 U^2+2 \eta\left\langle\zeta_t, \theta_t-\theta^\omega_\pi\right\rangle  
\end{align}
and
$$
\mb E^\omega_\pi[\|\theta_t-\theta^\omega_\pi\|_2^2] \leq \|\theta_0-\theta^\omega_\pi\|_2^2+\frac{U^2\tilde \mu^2}{64}.
$$
\end{lemma}

\begin{proof}
Observe first that by the construction of $\hat F$, we have
\begin{equation}\label{ineq:mse:L-F-xi}
    \|\hat{F}(\theta_t, \xi_t^\tau) - \hat{F}(\theta^\omega_\pi, \xi_t^\tau)\|_2
    =\|\phi(S_{t+\tau})- \gamma \phi(S_{t+\tau+1})\|_2\|\theta_t - \theta^\omega_\pi\|_2  \phi(S_{t+\tau})\|_2 \le 2\|\theta_t-\theta^\omega_\pi\|_2,
\end{equation}
and
\begin{equation}\label{ineq:mse:bounded-F-xi}
    \|\hat{F}(\theta^\omega_\pi, \xi_t^\tau)\|_2 = \|(\langle\phi(S_{t+\tau})-\gamma\phi(S_{t+\tau+1}), \theta_t\rangle-r(S_{t+\tau},A_{t+\tau})) \phi(S_{t+\tau})\|_2
    \le 2\|\theta^\omega_\pi\|_2 + 1 \le U.
\end{equation}
We then have
\begin{equation}\label{ineq:mse:cs1}
\begin{split}
    \langle \hat{F}(\theta_t, \xi_t^\tau), \theta_{t+1} - \theta_t \rangle 
    &= \langle \hat{F}(\theta_t, \xi_t^\tau) - \hat{F}(\theta^\omega_\pi, \xi_t^\tau) + \hat{F}(\theta^\omega_\pi, \xi_t^\tau), \theta_{t+1} - \theta_t \rangle 
    \\&\ge -2\|\theta_t-\theta^\omega_\pi\|_2 \| \theta_{t+1}-\theta_t\|_2 - U\|\theta_{t+1}-\theta_t\|_2,
\end{split}
\end{equation}
where the inequality follows from Cauchy-Schwarz inequality and the previous observations~\eqref{ineq:mse:L-F-xi} and~\eqref{ineq:mse:bounded-F-xi}.
We then obtain
\begin{equation}\label{ineq:mse:young2}
\begin{split}
    \eta \langle\hat{F}  (\theta_t, Z_t^\tau), \theta_{t+1}-\theta_t \rangle +\frac{1}{2}\|\theta_{t+1}-\theta_t\|_2^2     
    &\ge   \|\theta_{t+1}-\theta_t\|_2 \left(-\eta (2\|\theta_t-\theta^\omega_\pi\|_2+U) + \frac{1}{2}\|\theta_{t+1}-\theta_t\|_2\right) 
    \\& \ge   -\frac{1}{2}\eta^2 (2\|\theta_t-\theta^\omega_\pi\|_2+U)^2 
    \ge  - \eta^2 (4\|\theta_t-\theta^\omega_\pi\|_2^2 + U^2),
\end{split} 
\end{equation} 
where the first inequality follows from~\eqref{ineq:mse:cs1},  the second inequality holds 
\revise{since
$ bx - a x^2  \leq  b^2 /(4 a)$ for any $a>0$ and $x, b \in \RR$,}
and the third inequality uses $(a+b)^2 \le 2a^2 + 2b^2.$
Note that the update rule in Line~\ref{step:TD} of Algorithm~\ref{alg:TD} implies
\begin{equation*}
    \eta \langle \hat{F}(\theta_t, Z_t^\tau), \theta_t - \theta^\omega_\pi \rangle + \eta \langle \hat{F}(\theta_t, Z_t^\tau), \theta_{t+1}-\theta_t\rangle + \frac{1}{2} \|\theta_{t+1}-\theta_t\|_2^2= \frac{1}{2}\|\theta^\omega_\pi-\theta_t\|_2^2 - \frac{1}{2}\|\theta^\omega_\pi-\theta_{t+1}\|_2^2.
\end{equation*}
Lower bounding the left-hand side of the above expression using~\eqref{ineq:mse:young2} then yields
\begin{equation*}
    \eta \langle \hat{F}(\theta_t, \xi_t^\tau), \theta_t - \theta^\omega_\pi \rangle - 4 \eta^2 \|\theta_t - \theta^\omega_\pi\|_2^2 - U^2 \eta^2 \le \frac{1}{2}\|\theta^\omega_\pi-\theta_t\|_2^2 - \frac{1}{2}\|\theta^\omega_\pi-\theta_{t+1}\|_2^2.
\end{equation*}
Rearranging terms in the above expression further shows
\begin{equation*}
\begin{split}
   \frac{1}{2}\|\theta^\omega_\pi-\theta_{t+1}\|_2^2 
   &\le \left(\frac{1}{2} + 4 \eta^2\right) \|\theta_t-\theta^\omega_\pi\|_2^2 + U^2\eta^2 - \eta\langle\zeta_t,\theta_t-\theta^\omega_\pi\rangle -\eta\langle F(\theta_t), \theta_t - \theta^\omega_\pi\rangle
   \\&\le \left(\frac{1}{2} - \eta \mu  + 4 \eta^2\right) \|\theta_t-\theta^\omega_\pi\|_2^2 + U^2\eta^2 - \eta\langle\zeta_t,\theta_t-\theta^\omega_\pi\rangle ,
\end{split}
\end{equation*}
where the second inequality follows from Lemma~\ref{lemma-property-F}\ref{item:F:strong:monotone}.
The first claim then follows by noting that $\eta\le\tilde\mu /8$.
Taking expectation and multiplying by $2$ of the above expression yields
\begin{align*}
    \mb E^\omega_\pi [\|\theta^\omega_\pi-\theta_{t+1}\|_2^2 ] &\le \left(1 - 2\eta \mu  + 8 \eta^2\right) \|\theta_t-\theta^\omega_\pi\|_2^2 + 2U^2\eta^2 - 2\eta\langle\mb E^\omega_\pi [\zeta_t],\theta_t-\theta^\omega_\pi\rangle
    \\&\le \left(1 - 2\eta \tilde\mu  + 8 \eta^2\right) \|\theta_t-\theta^\omega_\pi\|_2^2 + 2U^2\eta^2,
\end{align*}
where the second inequality follows from Lemma~\ref{lemma:mixing}. 
Thus, the second claim also follows.
\end{proof}

\revise{Clearly, in view of Lemma \ref{lemma:TD-mse}, the third inequality in Condition \ref{assu-approx-H-V} is satisfied. 
Hence going forward we will focus on certifying the second inequality of Condition \ref{assu-approx-H-V}.
To this end,}
we proceed to establish the following high probability boundedness guarantee on the iterate~$\theta_t$ generated by Algorithm~\ref{alg:TD}. 
It should be noted that as the norm of the stochastic error $\zeta_t$ (defined in \eqref{def_td_err_zeta}) itself depends on the norm of the iterate $\theta_t$, \revise{one cannot simply take the telescopic sum of \eqref{td_norm_recursive} followed by invoking a standard off-the-shelf concentration argument to control the accumulation of stochastic errors (and subsequently the norm of $\theta_t$) in a high probability sense.}
In the following proposition, we construct an approximate martingale sequence that in turn will help us bound the accumulated stochastic error in stochastic approximation algorithms with high probability, 
which might be of independent interest.

\begin{proposition}
\label{prop-bounded-TD}
Suppose that Assumption~\ref{assu:geometric-ergodicity} holds.
Fix the total number of iterations $T>0$ and set $\theta_0=0$. For any $\delta \in(0,1)$, set 
$\tau> \max\{-\log(\sqrt{T}), \log(\mu)-\log(C(1+\gamma))\}/\log(\rho)$ 
and $\eta=\nu / \sqrt{T}$ 
with
$$
\begin{aligned}
& \nu = \min \left\{\frac{\tilde\mu }{L^2+8}, \frac{1}{2 U}, \frac{1}{4 G}\right\},
\ G=\max \left\{2 (1+\gamma) C (R^2+1), 4 \sqrt{\log \left(\frac{2T}{\delta}\right)}\left[(L+2)(R^2+1)+U (R+1)\right]\right\}.
\end{aligned}
$$
Then with probability at least $1-\delta$, we have
$$
\left\|\theta_t-\theta^\omega_\pi\right\|_2^2 \leq R^2+1, \quad\forall t \leq T .
$$
\end{proposition}

\begin{proof}
Define random sequences $\{X_t=\langle\zeta_t, \theta_t-\theta^\omega_\pi\rangle\},\{\widetilde{X}_t=\langle\zeta_t, \theta_t-\theta^\omega_\pi\rangle \mathds{1}_{\mathcal{G}_t}\}$, where $\mathcal{G}_t=\left\{Y_t \leq G \sqrt{t}\right\}$, and
$$
Y_0 = \widetilde{Y}_0 = 0, \ Y_t=Y_{t-1}+X_{t-1}, \ \tilde{Y}_t=\widetilde{Y}_{t-1}+\widetilde{X}_{t-1}.
$$
Conditioning on the event $\mathcal{G}_t$ and defining $M_{(\nu, G)}=R^2+2 \nu^2 U^2+2 \nu G$, recursively applying Lemma~\ref{lemma:TD-mse} and using the parameter choice of $\eta$ yields
\begin{align}\label{ineq-mse-thetabound}
    \|\theta_t-\theta^\omega_\pi\|_2^2 \leq R^2+2 t \eta^2 U^2+2 \eta G \sqrt{t}  =R^2+2 \nu^2 U^2+2 \nu G=M_{(\nu,G)},
\end{align}
We proceed to show that the above bound happens with probability at least $1-\delta$ given a proper choice of $(\nu, G)$.
First, note that
\begin{align*}
\langle\zeta_t, \theta_t-\theta^\omega_\pi\rangle & =\langle\hat{F}_t(\theta_t,\xi^\tau_t)-\hat{F}_t(\theta^\omega_\pi,\xi^\tau_t), \theta_t-\theta^\omega_\pi\rangle-\langle F(\theta_t), \theta_t-\theta^\omega_\pi\rangle+\langle\hat{F}_t(\theta^\omega_\pi,\xi^\tau_t), \theta_t-\theta^\omega_\pi\rangle \\
& \leq \|\hat{F}_t(\theta_t,\xi^\tau_t)-\hat{F}_t(\theta^\omega_\pi,\xi^\tau_t)\|_2 \|\theta_t-\theta^\omega_\pi\|_2+ \|F(\theta_t)\|_2 \|\theta_t-\theta^\omega_\pi\|_2 + \|\hat{F}_t(\theta^\omega_\pi,\xi^\tau_t)\|_2 \|\theta_t-\theta^\omega_\pi\|_2\\
& \leq(L+2)\|\theta_t-\theta^\omega_\pi\|_2^2+U\|\theta_t-\theta^\omega_\pi\|_2,
\end{align*}
where the first inequality follows from Cauchy-Schwarz inequality, and 
the second inequality uses the bounds~\eqref{ineq:mse:L-F-xi}, Lemma~\ref{lemma-property-F}\ref{item:F-ub}, and~\eqref{ineq:mse:bounded-F-xi}.
We then have
$$
\big|\widetilde{X}_t\big|=\left|\left\langle\zeta_t, \theta_t-\theta^\omega_\pi\right\rangle \mathds{1}_{\mathcal{G}_t}\right| \leq(L+2) M_{(\nu, G)}+U \sqrt{M_{(\nu, G)}}.
$$
In addition, observe that
\begin{equation}\label{ineq:high-prob-cond-e}
    \big|\mb E^\omega_\pi[\tilde X_t | \mac F_{t-1}]\big| = |\mb E^\omega_\pi[\langle \zeta_t, \theta_t-\theta^\omega_\pi\rangle]\mathds{1}_{\mac G_t}| \le (1+\gamma) C\rho^\tau\left\|\theta_t-\theta^\omega_\pi\right\|_2^2\mathds{1}_{\mac G_t}\le (1+\gamma) C\rho^\tau M_{(\nu,G)},
\end{equation}
where the equality holds because $\mathds{1}_{\cG_t}$ is $\cF_{t-1}$-measurable, the first inequality follows from Lemma~\ref{lemma:mixing}, and the second inequality uses~\eqref{ineq-mse-thetabound}.
Let $b=(L+2) M_{(\nu, G)}+U \sqrt{M_{(\nu, G)}}$. Chernoff bound then implies
$$
\begin{aligned}
\mathbb{P}\left(\tilde Y_t \geq x\right) & \leq \min _{\lambda>0} \exp (-\lambda x) \cdot \mb E^\omega_\pi\left[\exp \left(\sum_{i=0}^{t-1} \lambda\left(\tilde Y_{ i+1}-\tilde Y_i\right)\right)\right] \\
& = \min _{\lambda>0} \exp (-\lambda x) \cdot \mb E^\omega_\pi\left[\exp \left(\sum_{i=0}^{t-2} \lambda\left(\tilde Y_{ i+1}-\tilde Y_i\right)\right) \cdot \mb E^\omega_\pi\left[\exp\left(\lambda\left(\tilde Y_t- \tilde Y_{t-1}\right)\right) \mid \mathcal{F}_{t-1}\right]\right] \\
& \le \min _{\lambda>0} \exp (-\lambda x) \cdot \mb E^\omega_\pi\left[\exp \left(\sum_{i=0}^{t-2} \lambda\left(\tilde Y_{ i+1}-\tilde Y_i\right)\right)\right] \cdot \exp \left(\lambda (1+\gamma) C\rho^\tau M_{(\nu,G)}+\frac{b^2 \lambda^2}{2}\right) \\
& \le \min _{\lambda>0} \exp (-\lambda x) \cdot \exp \left(t \lambda (1+\gamma) C\rho^\tau M_{(\nu,G)}+\frac{t b^2 \lambda^2}{2}\right) \\
& =\exp \left(\min _{\lambda>0}-(x-t (1+\gamma) C\rho^\tau M_{(\nu,G)}) \lambda+\frac{t b^2 \lambda^2}{2}\right)=\exp \left(-\frac{(x-t (1+\gamma) C\rho^\tau M_{(\nu,G)})^2}{2 t b^2}\right),
\end{aligned}
$$
where the second inequality follows from Hoeffding's lemma and~\eqref{ineq:high-prob-cond-e}, and the third inequality holds by recursive application of the second inequality. Thus, for any $\delta \in(0,1)$, by applying the union bound over $1 \leq t \leq T$  in the above inequality, we have
\[\tilde Y_t \leq t (1+\gamma) C\rho^\tau M_{(\nu,G)}+2 \left((L+2) M_{(\nu, G)}+U \sqrt{M_{(\nu, G)}}\right) \sqrt{\log \left(\frac{2T}{\delta}\right) t}, \quad\forall t \in[T],\] 
with probability at least $1-\delta$. By construction of~$\tau$, $G$, and $\nu,$ we have $t (1+\gamma) C\rho^\tau M_{(\nu,G)}\le G\sqrt{t}/2$, $2 (L+2) M_{(\nu, G)} \sqrt{\log \left(\frac{2T}{\delta}\right) t} \le G\sqrt{t}/4$, together with $2 U \sqrt{M_{(\nu, G)}} \sqrt{\log \left(\frac{2T}{\delta}\right) t} \le G\sqrt{t}/4$. Thus,
$\tilde Y_t \leq G\sqrt{t}$ for all $ t \in[T]$ with probability $1-\delta$.
Denote by $\cG=\{\tilde Y_t \leq G\sqrt{t} \ \forall t\in[T]\}$.
We now use induction to show that $Y_t = \tilde Y_t $ over~$\mac G$ for all $t\le T.$ Note that the claim holds trivially at $t=0.$ 
Suppose that the claim holds at iteration $t\ge 0$, then for any $\omega\in\mac G$, we have
\begin{align*}
Y_{t+1}(\omega)   =Y_t(\omega)+X_t(\omega)  &=Y_t(\omega)+X_t(\omega) \mathds{1}_{\left\{\tilde{Y}_t \leq G \sqrt{t}\right\}}(\omega) \\
& = Y_t(\omega)+X_t(\omega) \mathds{1}_{\left\{Y_t \leq G \sqrt{t}\right\}}(\omega)=Y_t(\omega)+\widetilde{X}_t(\omega) =\widetilde{Y}_{t+1}(\omega),
\end{align*}
where the second equality follows from the definition of~$\mac G,$ the third equality applies the induction hypothesis, whereas the fourth and fifth equalities follow from the definitions of~$\tilde X_t$ and $\tilde Y_{t+1},$ respectively. The induction is then complete. We may now deduce that 
\begin{equation*}
    \|\theta_t-\theta^\omega_\pi\|_2^2 \le M_{(\nu,G)} = R^2+2 \nu^2 U^2+2 \nu G\leq R^2+1 \quad\forall t \leq T ,
\end{equation*}
where the first inequality follows from the induction claim, the equality holds by definition of $M_{(\nu,G)}$, and the second inequality uses the choice of~$\nu$ that $\nu\le \min\{1/(2U),1/(4G)\}.$ Hence, the claim follows.
\end{proof}

We are now ready to show that with a proper specification of parameters $(\eta, T)$, Algorithm~\ref{alg:TD} indeed produces a solution that certifies  Condition~\ref{assu-approx-H-V}.
\begin{theorem}\label{thm-TD-V}
Under the same setup of Proposition~\ref{prop-bounded-TD}, with probability at least $1-\delta$, the following hold for all $t\le T$.
\begin{enumerate}[label = (\roman*)]
    \item \label{item-bias-TD-V} $\|\mb E^\omega_\pi[Q_{\theta_t}]-Q^\omega_\pi\|_{\infty} \leq(1-\eta \tilde\mu )^{t / 2}R+\varepsilon_{\rm{approx}}$;
    \item \label{item-bounded-TD-V} $\| Q_{\theta_t} \|_\infty \le R+1+\varepsilon_{\rm{approx}}$;
    \item \label{item-mse-TD-V} $\displaystyle \mb E^\omega_\pi[\|Q_{\theta_t}-Q^\omega_\pi\|_\infty^2] \le R^2 +\frac{U^2\tilde \mu^2}{64} + \varepsilon_{\rm{approx}}^2$.
\end{enumerate}
\end{theorem}

\begin{proof}
For Assertion~\ref{item-bias-TD-V}, we have
\begin{equation*}
\begin{split}
    \|\mb E^\omega_\pi[Q_{\theta_t}]-Q^\omega_\pi\|_{\infty} &\le \|\mb E^\omega_\pi[Q_{\theta_t}]-Q_{\theta^\omega_\pi}\|_{\infty} + \|Q_{\theta^\omega_\pi}-Q^\omega_\pi\|_{\infty} 
    \\&\le \sup_{z\in\cZ}\|\phi(z)^\top \mb E^\omega_\pi[\theta_t] - \phi(z)^\top \theta^\omega_\pi\|_2 + \varepsilon_{\rm{approx}} 
    \\&\leq \sup_{z\in\cZ}\|\phi(z)\|_2 \|\mb E^\omega_\pi[\theta_t] - \theta^\omega_\pi\|_2 + \varepsilon_{\rm{approx}} \leq (1-\eta \tilde\mu )^{t / 2}R+\varepsilon_{\rm{approx}},
\end{split}
\end{equation*}
where the second inequality holds by the definition of $\varepsilon_{\rm{approx}}$, and the last inequality follows from the assumption that $\|\phi(z)\|\le 1$ for all $z\in\cZ$ and Lemma~\ref{lemma-bias-conv}.
For Assertion~\ref{item-bounded-TD-V}, we have
\begin{equation*}
\begin{split}
   \|Q_{\theta_t} \|_\infty 
   &\le \|Q_{\theta_t}-Q_{\theta^\omega_\pi} \|_\infty + \|Q_{\theta^\omega_\pi}-Q^\omega_\pi \|_\infty 
    \\&\le \sup_{z\in\cZ}\|\phi(z)^\top \theta_t - \phi(z)^\top \theta^\omega_\pi\|_2 + \varepsilon_{\rm{approx}} 
   \\&\le \sup_{z\in\cZ}\|\phi(z)\|\|\theta_t-\theta^\omega_\pi\|_2 + \varepsilon_{\rm{approx}} 
   \le \sqrt{R^2+1} + \varepsilon_{\rm{approx}} \le R+1+\varepsilon_{\rm{approx}},
\end{split}
\end{equation*}
where the fourth inequality follows from Lemma~\ref{prop-bounded-TD}. 
For Assertion~\ref{item-mse-TD-V}, observe that 
\begin{equation*}
\begin{split}
    \mb E^\omega_\pi[\|Q_{\theta_t}-Q^\omega_\pi\|_\infty^2] \le 2 (\mb E^\omega_\pi [\|\theta^\omega_\pi-\theta_t\|_2^2 ]+\varepsilon_{\rm{approx}}^2) \le R^2 + \frac{U^2\tilde \mu^2}{64} + \varepsilon_{\rm{approx}}^2,
\end{split}
\end{equation*}
where the second inequality uses Lemma~\ref{lemma:TD-mse}.
Hence, the claim follows.
\end{proof}

%% file: robust_eval_meta.tex

\subsection{Sample Complexity for Robust Policy Evaluation}\label{ssec:putting-it-together}
\revise{
With Theorem \ref{thm-V-spda} and Theorem \ref{thm-TD-V} in place, we  are now ready to determine the total number of samples required by Algorithm~\ref{alg:SPDA} to estimate the robust value function $V_\pi$ up to a pre-specified target precision $\varepsilon$ with high probability. 
}

\begin{theorem}
\label{thm:altogether-spda}
Suppose that Assumption~\ref{assu:geometric-ergodicity} holds, fix the total number of iterations $M>0$ in Algorithm~\ref{alg:SPDA} and $\delta\in(0,1)$. Denote $B=R+1+\varepsilon_{\rm{approx}}$. For any
$\varepsilon\ge 8 \varepsilon_{\rm{approx}}/(1-\gamma)$, 
set the parameters of Algorithm~\ref{alg:SPDA} as
$$
\alpha_m=\sqrt{m}, \quad \lambda_m=\frac{(m+1) B }{2 \sqrt{\log(|\cK|)}}, \quad \forall m \leq M,
$$ 
where Algorithm~\ref{alg:TD} is instantiated with parameters
\[\tau> \max\{-\log(\sqrt{T}), \log(\mu)-\log(C(1+\gamma))\}/\log(\rho),\ \eta=\nu/\sqrt{T},\ \theta_0=0, \ T=\cO\left(\frac{\log^2(B/\varepsilon)}{\nu^2\tilde\mu^2}\right),\]
with 
\begin{align*}
\nu = \min\! \left\{\frac{\tilde\mu }{L^2+8}, \frac{1}{2 U}, \frac{1}{4 G}\right\}\!,\,
 G=\max \left\{2 (1+\gamma) C (R^2+1), 4 \sqrt{\log \left(\!\frac{8MT}{\delta}\right)}\left[(L+2)(R^2+1)+U (R+1)\right]\right\}\!.
\end{align*}
The total number of samples required by Algorithm~\ref{alg:SPDA} to output
\begin{align*}
   -\varepsilon\le \hat V_\pi(s)-V_\pi(s) \le \varepsilon,\quad 
   -\varepsilon\le \hat Q_\pi(s,a,k)-Q_\pi(s,a,k) \le \varepsilon \quad \forall (s,a,k)\in\cS\times\cA\times\cK
\end{align*}
with probability at least $1-\delta$ is then bounded by
\begin{align*}
\widetilde{\mathcal O}\left(\left(\frac{\bar B^{2}\log(|\cK|)}{(1-\gamma)^{2}}\log\rbr{\frac{\bar B^2}{\delta(1-\gamma)^2\varepsilon^2}}+J\delta\right)\frac{1}{\tilde\mu^2\nu^2\varepsilon^2}\right).
\end{align*}
\end{theorem}

\begin{proof}
Note that the parameter choice clearly satisfies the conditions of Theorem~\ref{thm-TD-V}, which in turn certifies the conditions of Theorem~\ref{thm-V-spda}. In addition, observe that the total number of samples required is~$MT$, where 
\begin{align*}
M=\cO\left(\frac{\bar B^{2}\log(|\cK|)}{(1-\gamma)^{2} \varepsilon^2}\log\rbr{\frac{\bar B^2}{\delta(1-\gamma)^2\varepsilon^2}}+\frac{J\delta }{\varepsilon^2}\right)
\end{align*}
is given by Theorem~\ref{thm-V-spda}. Thus, the claim follows.
\end{proof}

A few remarks are in order before we conclude this section. 
First, in view of Theorem \ref{thm:altogether-spda}, 
Algorithm \ref{alg:SPDA} requires $\tilde{\cO}(1/\varepsilon^2)$ samples to estimate the robust value function up to $\varepsilon$ accuracy.
In addition, the obtained accuracy certificate is stated in a high probability sense. 
This also appears to be the first robust policy evaluation method for continuous-state robust MDPs with optimal sample complexity guarantees. 
Notably, in contrast to \cite{zhou2023natural,tamar2014scaling}, the proposed Algorithm \ref{alg:SPDA} does not require any restrictive assumption on the diameter of $\cP$ or the discount factor. 
Second, although we consider controller's action space to be finite within this section, results for robust policy evaluation developed here can be naturally  extended to  continuous action space with minor changes.
Finally, in Section \ref{sec:RPO}, we will utilize Algorithm \ref{alg:SPDA} as a subroutine, and  introduce an efficient policy optimization method for \eqref{def_rmdp_general} that obtains the optimal sample complexity up to a logarithmic factor.

%% file: robust_opt.tex

\section{Robust Policy Optimization} \label{sec:RPO}

\begin{algorithm}[b!]
\caption{Approximate policy iteration}
\begin{algorithmic}[1]
\REQUIRE Initial controller's policy $\pi_0\in\Pi$, number of iterations $N$, 
error tolerance $\Delta,\varepsilon>0$, 
confidence level $ 1- \delta $
\FOR{$n=0,1,\dots,N$}
  \STATE \label{step:matrix-game}
  {Find an approximate solution $(\pi_{n+1}(\cdot|s), \omega_{n+1}(\cdot|s))$ of~\eqref{matrix-game-concept}}
  such that $ \hatQsvp{\pi_n}{s}{\omega_{n+1}}{\pi} - \hatQsvp{\pi_n}{s}{\omega}{\pi_{n+1}}\le \Delta$ for all $(\pi(\cdot|s),\omega(\cdot|s))\in \Delta_\cA \times \Delta_\cK$ 
  \STATE \label{step:approximate-rpe}
  Compute $\hat Q_{\pi_{n+1}} $  using Algorithm~\ref{alg:SPDA} with input $\pi_{n+1}$
  such that $|\hat Q_{\pi_{n+1}}(s,a,k)-Q_{\pi_{n+1}}(s,a,k)| \le \varepsilon$ for all $(s,a,k)\in\cS\times\cA\times\cK$ with probability $1-\delta/(2N)$ 
\ENDFOR
\RETURN $\pi_{N+1}$
\end{algorithmic}\label{alg:api}
\end{algorithm}

With the development of robust policy evaluation in Section \ref{sec:RPE}, we are now ready to introduce the proposed method for the robust policy optimization problem \eqref{def_rmdp_general}.
In view of the dynamic game formulation~\eqref{def_dynamic_game} in Section~\ref{sec_ambiguity_set}, 
the robust policy optimization problem~\eqref{def_rmdp_general} is equivalent to a dynamic zero-sum game between the controller and nature. 
In Algorithm \ref{alg:api}, we present an approximate policy iteration method for solving~\eqref{def_dynamic_game}.
At the $n$-th iteration, the proposed method takes the form of 
\begin{equation}\label{matrix-game-concept}
   (\pi_{n+1}(\cdot|s), \omega_{n+1}(\cdot|s))\; {\text{ solves }}  \max_{\pi(\cdot|s)\in\Delta_{\cA}} \min_{\omega(\cdot|s)\in \Delta_{\cK}} (\pi(\cdot|s))^\top \hat{Q}_{\pi_n}(s)\omega(\cdot|s) \quad\forall s\in\cS,
\end{equation}
where $\hat{Q}_{\pi_n}$ denotes an estimate of the robust joint action-value function defined in \eqref{eq_def_opt_joint_q_give_pi}, and with a slight overload of notation we use  $\hat{Q}_{\pi_n}(s)\in\R^{|\cA|\times |\cK|}$ to denote the matrix defined as 
$\hat{Q}_{\pi_n}(s)[a,k] = \hat{Q}_{\pi_n}(s,a,k)$ for $(s,a,k) \in \cS \times \cA \times \cK$.
It is worth noting here that the estimator $\hat{Q}_{\pi_n}$ is constructed by invoking Algorithm~\ref{alg:SPDA} \revise{with a precision target $\varepsilon > 0$} and parameters specified as in Theorem~\ref{thm:altogether-spda}.
In addition, we assume access to a computational oracle that can produce $ (\pi_{n+1}(\cdot|s), \omega_{n+1}(\cdot|s))$ whose duality gap is upper bounded by a prespecified precision $\Delta > 0$. 
Such an oracle can be instantiated by various off-the-shelf methods \cite{nemirovski2004prox, lan2023novel,nesterov2005smooth,chambolle2011first}.
\revise{We will proceed to establish some general convergence properties of Algorithm~\ref{alg:api},
which will be used later in Theorem \ref{total_sample_complexity_opt} to determine the total iteration and sample complexities for finding an $\epsilon$-optimal policy for the robust policy optimization problem \eqref{def_rmdp_general}.}
Note that similar to \eqref{line:ftrl} in Section \ref{sec:RPE}, update \eqref{matrix-game-concept} here defines access to the updated controller's policy $\pi_{n+1}$ via a bilinear matrix game, and hence there is no need for explicitly solving \eqref{matrix-game-concept} for every state $s \in \cS$. 
Instead, $\pi_{n+1}(\cdot|s)$ is generated only when queried at state $s \in \cS$ by Algorithm \ref{alg:SPDA} in the robust policy evaluation step.

We begin by summarizing several useful properties of the iterates generated by Algorithm~\ref{alg:api}.
To facilitate our discussion, let us denote $\omega^\star_n $ as the optimal policy of nature's cost-minimizing MDP given controller's policy~$\pi_n$. The existence of $\omega^\star_n$ follows from the dynamic equations of \eqref{def_dynamic_game}. Clearly, $\omega^\star_n$ can be chosen independent of the initial state in \eqref{def_dynamic_game}.
\begin{lemma}
\label{lemma:api-descent}
Fix the total number of iterations $N>0$ in Algorithm~\ref{alg:api} and let $\delta\in(0,1)$. At each iteration of Algorithm~\ref{alg:api}, the following hold for all $s\in\cS$ with probability $1-\delta/(2N)$.
\begin{enumerate}[label = (\roman*)]
    \item \label{item-game-approx-error-Q} $\big| \hatQsvp{\pi_n}{s}{\omega}{\pi} - (\pi(\cdot|s))^\top Q_{\pi_n}(s)\,\omega(\cdot|s) \big| \le \varepsilon$  for all $(\pi,\omega)\in\Pi\times\Omega$;
    \item \label{item-game-real-Q-opt-error} $\Qsvp{\pi_n}{s}{\omega_{n+1}}{\pi} - \Qsvp{\pi_n}{s}{\omega}{\pi_{n+1}} \le \Delta+ 2\varepsilon $ for all $(\pi,\omega)\in\Pi\times\Omega$;
    \item \label{item-game-real-policy-progress} $\Qsvp{\pi_n}{s}{\omega^\star_{n+1}}{\pi_{n+1}}-\Qsvp{\pi_n}{s} {\omega^\star_n}{\pi_n} + \Delta + 2\varepsilon \ge 0$.
\end{enumerate}
\end{lemma}
\begin{proof}
For Assertion~\ref{item-game-approx-error-Q}, we have with probability $1-\delta/(2N)$ that
\begin{align*}
& \left| (\pi(\cdot|s))^\top \hat Q_{\pi_n}(s)\omega(\cdot|s) - (\pi(\cdot|s))^\top Q_{\pi_n}(s)\omega(\cdot|s) \right|
\\&= \left|\sum_{a\in\cA}\sum_{k\in\cK} \left(\hat Q_{\pi_n}(s,a,k) - Q_{\pi_n}(s,a,k)\right)\omega(k|s) \pi(a|s) \right|
\le \varepsilon  \quad \forall (\pi,\omega)\in\Pi\times\Omega,\ \forall s\in\cS,
\end{align*}
where the inequality readily follows from the condition in Line~\ref{step:approximate-rpe} of Algorithm~\ref{alg:api}, which is in turn satisfied by the output of Algorithm~\ref{alg:SPDA} (\emph{cf}. Theorem~\ref{thm:altogether-spda}).
For Assertion~\ref{item-game-real-Q-opt-error}, we have with probability $1-\delta/(2N)$~that
\begin{equation*}
\begin{split}
    & \Qsvp{\pi_n}{s}{\omega_{n+1}}{\pi} - \Qsvp{\pi_n}{s}{\omega}{\pi_{n+1}}
    \\&\le \hatQsvp{\pi_n}{s}{\omega_{n+1}}{\pi}  - \hatQsvp{\pi_n}{s}{\omega}{\pi_{n+1}} + 2\varepsilon \le \Delta+ 2 \varepsilon \quad \forall (\pi,\omega)\in\Pi\times\Omega,\ \forall s\in\cS,
\end{split}
\end{equation*}
where the first inequality follows from Assertion~\ref{item-game-approx-error-Q}, and the second inequality applies the condition in Line~\ref{step:matrix-game} of Algorithm~\ref{alg:api}.
For Assertion~\ref{item-game-real-policy-progress}, we have with probability $1-\delta/(2N)$ that
\begin{align*}
    \Qsvp{\pi_n}{s}{\omega^\star_n}{\pi_n} & \le \Qsvp{\pi_n}{s}{\omega_{n+1}}{\pi_n} 
    \\&\le \Qsvp{\pi_n}{s}{\omega^\star_{n+1}}{\pi_{n+1}} + \Delta + 2\varepsilon \quad \forall s\in\cS,
\end{align*}
where the first inequality holds by optimality of~$\omega^\star_n$, and the second inequality follows by taking $(\pi,\omega)=(\pi_{n},\omega^\star_{n+1})$ in Assertion~\ref{item-game-real-Q-opt-error}.
\end{proof}

\begin{remark}\label{remark_high_prob}
    It is interesting to note here that for the purpose of the convergence characterization for Algorithm~\ref{alg:api}, one has to control the estimation error $\Vert\hat Q_{\pi_{n}} -Q_{\pi_{n}}\Vert_\infty $ either in expectation, or with high probability as stated in Line~\ref{step:approximate-rpe}. 
    \revise{
    Despite the abundant development of non-robust policy evaluation \cite{kotsalis2022simple, sutton1988learning}, it is unclear to us how existing  development could be readily adapted for controlling $\Vert\hat Q_{\pi_{n}} -Q_{\pi_{n}}\Vert_\infty $  in expectation.}
    \revise{Instead, in this manuscript we proceed} with Theorem \ref{thm:altogether-spda} that certifies high probability control.
   It can also be readily seen that Lemma \ref{lemma:api-descent} and the ensuing Lemma \ref{lemma-per-iter-gap-api} no longer hold for approximate policy iteration if one can only control the bias term  $\EE [{ \hat Q_{\pi_{n}} - Q_{\pi_n}}]$. This for instance is the approach taken in \cite{alacaoglu2022natural}. 
\end{remark}

We proceed by characterizing the difference of value functions between any pair of joint policies. 
This could be understood as the performance difference lemma \cite{kakade2002approximately} in the dynamic game setting. 

\begin{lemma}
\label{lemma:primal-dual-pdl}
For any joint policies $(\pi,\omega),(\pi', \omega')\in \Pi\times\Omega$ and any $s\in\cS$, we have
\begin{align*}
    V^\omega_\pi(s)-V^{\omega'}_{\pi'}(s)=\frac{1}{1-\gamma} \int_{\cS} \!\left[ \Qwsvp{\omega'}{\pi'}{s'}{\omega}{\pi}- \Qwsvp{\omega'}{\pi'}{s'}{\omega'}{\pi'}\right] d^\omega_\pi(\diff s'|s).
\end{align*}
\end{lemma}
\begin{proof}
We have
\begin{align*}
    V^\omega_\pi(s)-V^{\omega'}_{\pi'}(s) &= \Qwsvp{\omega}{\pi}{s}{\omega}{\pi} - \Qwsvp{\omega'}{\pi'}{s}{\omega'}{\pi'}
    \\&= (\pi(\cdot|s))^\top \rbr{Q^\omega_\pi(s) - Q^{\omega'}_{\pi'}(s)} \omega(\cdot|s) 
        + \Qwsvp{\omega'}{\pi'}{s}{\omega}{\pi} - \Qwsvp{\omega'}{\pi'}{s}{\omega'}{\pi'}
    \\&=\gamma \sum_{a\in\cA}\sum_{k\in\cK}\int_{\cS}[V^\omega_\pi(s')-V^{\omega'}_{\pi'}(s')]\pi(a|s)\omega(k|s)P_k(\diff s'|s,a)
    \\&\quad ~ + \Qwsvp{\omega'}{\pi'}{s}{\omega}{\pi} - \Qwsvp{\omega'}{\pi'}{s}{\omega'}{\pi'},
\end{align*}
where the third equality follows from Definition~\ref{def_nature_state_action_func}.
The claim then follows by iteratively expanding $V^\omega_\pi(s')-V^{\omega'}_{\pi'}(s')$ and using Definition~\ref{def_discounted_visitation}.
\end{proof}
We are now ready to establish the following convergence characterization for each update of Algorithm~\ref{alg:api}.
\revise{
In particular, Lemma \ref{lemma-per-iter-gap-api} below establishes that the optimality gap of the policy generated by Algorithm~\ref{alg:api} contracts linearly up to an additive error proportional to the robust policy evaluation error $\varepsilon$, and the primal-dual gap of the approximate solution $(\pi_{n+1}, \omega_{n+1})$ for the matrix game~\eqref{matrix-game-concept}.
}
\begin{lemma}\label{lemma-per-iter-gap-api}
Fix total iteration number $N>0$ in Algorithm~\ref{alg:api} and set $\delta\in(0,1)$. 
Let $\rho\in\cP(\cS)$ with $\mathrm{supp}(\rho)=\cS$, and \revise{let  $D=\sup_{(\pi,\omega)\in\Pi\times\Omega}\!\|\!\int_{\cS}d^\omega_{\pi}(\cdot|s)\rho(\diff s)/\rho(\cdot)\|_\infty$ denote the distribution mismatch coefficient.}
Then, with probability at least $1-\delta/N$,
\begin{align*}
\int_{\cS}\left[V_{\pi^\star}(s)-V_{\pi_{n+1}}(s)\right]\rho(\diff s) &\!\le\! \left(1-\frac{1-\gamma}{D}\right)\int_{\cS}\left[V_{\pi^\star}(s)-V_{\pi_n}(s)\right]\rho(\diff s)
 + \frac{2D(\Delta+2\varepsilon)}{1-\gamma}.
\end{align*}
\end{lemma}
    
\begin{proof}
From Lemma~\ref{lemma:primal-dual-pdl}, we have with probability $1-\delta/(2N)$ that
\begin{equation}\label{ineq:nonneg}
\begin{aligned}
    &V_{\pi_{n+1}}(s)-V_{\pi_n}(s)+\frac{1}{1-\gamma}(\Delta+2\varepsilon) 
    \\&=\frac{1}{1-\gamma} \int_{\cS} \Big[\Qsvp{\pi_n}{s'}{\omega^\star_{n+1}}{\pi_{n+1}}-\Qsvp{\pi_n}{s'}{\omega^\star_n}{\pi_n}
    +\Delta + 2\varepsilon\Big] d^{\omega^\star_{n+1}}_{\pi_{n+1}}(\diff s'|s)
    \\&\ge \frac{d^{\omega^\star_{n+1}}_{\pi_{n+1}}(\{s\}|s)}{1-\gamma} \Big[\Qsvp{\pi_n}{s'}{\omega^\star_{n+1}}{\pi_{n+1}}-\Qsvp{\pi_n}{s'}{\omega^\star_n}{\pi_n}
    +\Delta + 2\varepsilon\Big]
    \\&\ge \Qsvp{\pi_n}{s'}{\omega^\star_{n+1}}{\pi_{n+1}}-\Qsvp{\pi_n}{s'}{\omega^\star_n}{\pi_n}
        +\Delta + 2\varepsilon
    \ge 0,
\end{aligned}
\end{equation}
where the first inequality applies Lemma~\ref{lemma:api-descent}, and the second inequality follows from the fact that $d^{\omega^\star_{n+1}}_{\pi_{n+1}}(\{s\}|s)$ $\ge1-\gamma$. 
Consequently, integrating both sides of the above inequality  with respect to $d^{\omega_{n+1}}_{\pi^\star}(\cdot|s)$ yields
\begin{equation}\label{eq:api-per-step}
    \begin{split}
        &\int_{\cS}\left[V_{\pi_n}(s')-V_{\pi_{n+1}}(s')\right]d^{\omega_{n+1}}_{\pi^\star}(\diff s'|s)
        \\&\le \int_{\cS}\left[\Qsvp{\pi_n}{s'}{\omega^\star_n}{\pi_n} - \Qsvp{\pi_n}{s'}{\omega^\star_{n+1}}{\pi_{n+1}}\right]d^{\omega_{n+1}}_{\pi^\star}(\diff s'|s)
        +\frac{1}{1-\gamma}(\Delta+2\varepsilon).
    \end{split}
\end{equation}
On the other hand, we have 
\begin{equation}\label{eq:api-subopt-gap}
\begin{split}
(1-\gamma)(V_{\pi^\star}(s) - V_{\pi_n}(s))
&\le (1-\gamma)(V^{\omega_{n+1}}_{\pi^\star}(s) - V_{\pi_n}(s))
\\&=\int_{\cS} \left[\Qsvp{\pi_n}{s'}{\omega_{n+1}}{\pi^\star} - \Qsvp{\pi_n}{s'}{\omega^\star_n}{\pi_n}\right] d^{\omega_{n+1}}_{\pi^\star}(\diff s'|s),
\end{split}
\end{equation}
where the inequality holds from the definition of $V_{\pi^\star}(s)$, and the equality follows from Lemma~\ref{lemma:primal-dual-pdl}.
Summing up~\eqref{eq:api-per-step} and~\eqref{eq:api-subopt-gap} and rearranging terms, we obtain with probability $1-\delta/N$ that
\begin{equation}\label{ineq:api-first-bound}
\begin{split}
&\int_{\cS}\left[V_{\pi_n}(s')-V_{\pi_{n+1}}(s')\right]d^{\omega_{n+1}}_{\pi^\star}(\diff s'|s) + (1-\gamma)(V_{\pi^\star}(s) - V_{\pi_n}(s))
\\&\le \int_{\cS}\big[\Qsvp{\pi_n}{s'}{\omega_{n+1}}{\pi^\star} - \Qsvp{\pi_n}{s'}{\omega^\star_{n+1}}{\pi_{n+1}}\big]d^{\omega_{n+1}}_{\pi^\star}(\diff s'|s)
+\frac{1}{1-\gamma}(\Delta+2\varepsilon)
\\&\le \Delta+2\varepsilon+\frac{1}{1-\gamma}(\Delta+2\varepsilon),
\end{split}
\end{equation}
where the second inequality follows from the observation that
\begin{align*}
&\Qsvp{\pi_n}{s}{\omega_{n+1}}{\pi^\star} - \Qsvp{\pi_n}{s}{\omega^\star_{n+1}}{\pi_{n+1}} \le \Delta+2\varepsilon \quad \forall s\in\cS,
\end{align*}
which in turn follows from Lemma~\ref{lemma:api-descent}\ref{item-game-real-Q-opt-error} applied to $(\pi,\omega)=(\pi^\star,\omega^\star_{n+1})$.
Finally, we obtain
\begin{align*}
& D \int_{\cS}\left[V_{\pi_n}(s)-V_{\pi_{n+1}}(s)\right] \rho(\diff s) +(1-\gamma)\int_{\cS}\left[V_{\pi^\star}(s) - V_{\pi_n}(s)\right] \rho(\diff s)
    \\&=D \int_{\cS}\left[V_{\pi_n}(s)-V_{\pi_{n+1}}(s)-\frac{1}{1-\gamma}(\Delta+2\varepsilon)\right] \rho(\diff s) + \frac{D}{1-\gamma}(\Delta+2\varepsilon)+(1-\gamma)\int_{\cS}\left[V_{\pi^\star}(s) - V_{\pi_n}(s)\right]\rho(\diff s)
    \\&\le \iint_{\cS\times\cS}\left[V_{\pi_n}(s')-V_{\pi_{n+1}}(s')-\frac{1}{1-\gamma}(\Delta+2\varepsilon)\right] d^{\omega_{n+1}}_{\pi^\star}(\diff s'|s)\rho(\diff s)
    \\&\quad+ \frac{D}{1-\gamma}(\Delta+2\varepsilon)
    +(1-\gamma)\int_{\cS}\left[V_{\pi^\star}(s) - V_{\pi_n}(s)\right]\rho(\diff s)
    \\&= \int_{\cS}\left[\int_{\cS}\left[V_{\pi_n}(s')-V_{\pi_{n+1}}(s')\right]d^{\omega_{n+1}}_{\pi^\star}(\diff s'|s) + (1-\gamma)(V_{\pi^\star}(s) - V_{\pi_n}(s))\right]\rho(\diff s)
    \\&\quad-\frac{1}{1-\gamma}(\Delta+2\varepsilon)+ \frac{D}{1-\gamma}(\Delta+2\varepsilon)
    \\&= \int_{\cS}\sbr{\Delta+2\varepsilon+\frac{1}{1-\gamma}(\Delta+2\varepsilon)} \rho(\diff s)-\frac{1}{1-\gamma}(\Delta+2\varepsilon)+ \frac{D}{1-\gamma}(\Delta+2\varepsilon)
    \\&\le \Delta+2\varepsilon+\frac{D}{1-\gamma}(\Delta+2\varepsilon),
\end{align*}
where the first inequality follows from
\eqref{ineq:nonneg} and $D\ge 1$ by construction, the second inequality holds because of~\eqref{ineq:api-first-bound}.
The claim thus follows after dividing both sides in the above expression by~$D$ and rearranging terms.
\end{proof}

\revise{With \revise{Theorem \ref{thm:altogether-spda}} and  Lemma~\ref{lemma-per-iter-gap-api} in place, we are now ready to  instantiate Algorithm \ref{alg:api} with concrete specifications of $\varepsilon$ and $\Delta$}, and correspondingly determine its \revise{iteration and} sample complexities  for the robust policy optimization problem \eqref{def_rmdp_general}. 
\begin{theorem}\label{total_sample_complexity_opt}
\label{thm-sample-complexity-api}
Let $\delta\in(0,1)$. Suppose $\rho\in\cP(\cS)$ is such that $\mathrm{supp}(\rho)=\cS$, and define $D=\sup_{(\pi,\omega)\in\Pi\times\Omega}$ $\!\|\!\int_{\cS}d^\omega_{\pi}(\cdot|s)\rho(\diff s)/\rho(\cdot)\|_\infty$. 
\revise{For any $\epsilon \ge 128 D^2 (1-\gamma)^{-3}\varepsilon_{\mathrm{approx}}$ and any $\pi_0\in\Pi$,
run Algorithm~\ref{alg:api} for
$N=\cO\left(D(1-\gamma)^{-1}\log\left(((1-\gamma)\epsilon)^{-1}\right)\right)$ iterations with $\Delta = (1-\gamma)^2\epsilon/(8 D^2)$ and $\varepsilon = (1-\gamma)^2\epsilon/(16D^2)$.}
Then, with probability at least $1-\delta$, the output~$\pi_N$ of Algorithm~\ref{alg:api} satisfies
\begin{align*}
\int_{\cS}V_{\pi^\star}(s)\rho(\diff s)-\int_{\cS}V_{\pi_N}(s)\rho(\diff s) \le \epsilon.
\end{align*}
In addition, the total number of samples required by Algorithm~\ref{alg:api} is bounded by
\begin{align*}
\widetilde{\mathcal O}\left(\left(\frac{\bar B^{2}\log(|\cK|)}{(1-\gamma)^{2}}\log\rbr{\frac{\bar B^2}{\delta(1-\gamma)^2\epsilon^2}}+J\delta\right)\frac{1}{\mu^2\nu^2\epsilon^2}\right).
\end{align*}
\end{theorem}

\begin{proof}
We have
\begin{align*}
\int_{\cS}\left[V_{\pi^\star}(s)-V_{\pi_{N}}(s)\right]\rho(\diff s) &\le \left(1-\frac{1-\gamma}{D}\right)^{\!N}\int_{\cS}\left[V_{\pi^\star}(s) - V_{\pi_0}(s)\right]\rho(\diff s) + \frac{2D^2(\Delta+2\varepsilon)}{(1-\gamma)^2}\left(1-\left(1-\frac{1-\gamma}{D}\right)^{\!N}\right)
\\&\le \left(1-\frac{1-\gamma}{D}\right)^{\!N}\int_{\cS}\left[V_{\pi^\star}(s) - V_{\pi_0}(s)\right]\rho(\diff s) + \frac{2D^2(\Delta+2\varepsilon)}{(1-\gamma)^2} 
\le \epsilon,
\end{align*}
where the first inequality follows by a recursive application of Lemma~\ref{lemma-per-iter-gap-api}, 
and the third inequality uses the choice of $N$ and the choices of $\varepsilon$ and $\Delta$. The claim then follows by invoking Theorem~\ref{thm:altogether-spda} and the choice of~$\varepsilon$,  which states that each iteration of Algorithm~\ref{alg:api} requires  at most
\begin{align*}
\widetilde{\mathcal O}\left(\left(\frac{\bar B^{2}\log(|\cK|)}{(1-\gamma)^{2}}\log\rbr{\frac{\bar B^2}{\delta(1-\gamma)^2\varepsilon^2}}+J\delta\right)\frac{1}{\mu^2\nu^2\varepsilon^2}\right)
\end{align*}
samples.
\end{proof}

In view of Theorem~\ref{thm-sample-complexity-api}, the number of samples required by Algorithm~\ref{alg:api} to find an $\epsilon$-optimal policy for the robust policy optimization problem \eqref{def_rmdp_general} is bounded by $\tilde{\cO}(1/\epsilon^2)$.
Clearly, the obtained sample complexity bound is order-wise optimal up to a logarithmic factor \cite{sidford2018near}. 
Notably the accuracy certificate of Algorithm~\ref{alg:api} is also stated in a high probability sense.

\begin{remark}
Before we conclude our discussion, it might be worth noting here 
that the proposed methods in Sections~\ref{sec:RPE} and~\ref{sec:RPO} can be straightforwardly adapted for solving general infinite-horizon  zero-sum Markov games with continuous state and finite action spaces, 
and all the performance guarantees we obtained extend directly to this setting. 
Zero-sum Markov games have received considerable recent attention.
When the  model is known, dynamic programming methods and more recently policy gradient methods have been studied in \cite{zeng2022regularized,perolat2015approximate,cen2022faster,song2023can,van1978discounted,Patek1997}.
When the model is unknown, optimal $\cO(1/\epsilon^2)$ sample complexities have been established with model-based methods \cite{li2022minimax,zhang2023model,yan2024model} and model-free methods  \cite{alacaoglu2022natural,chen2023finite}.
On the other hand, it is unclear how the same sample complexity bound could be obtained for continuous-state MDPs, as existing optimal methods do not admit direct generalization to continuous state spaces. 
There also exists a fruitful line of research on developing model-free methods for solving zero-sum Markov games with continuous state spaces.
Yet it appears that the current development either obtains suboptimal sample complexity \cite{zhao2022provably} except for finite horizon problems \cite{jin2022power,huang2021towards}, or requires additional assumptions on the structure of the model \cite{xie2020learning,chen2022almost}.
In this sense, the proposed method in this manuscript seems to be the first algorithm for general infinite-horizon zero-sum  Markov games with order-wise optimal sample complexity of $\tilde{\cO}(1/\epsilon^2)$ for continuous state spaces. 
\end{remark}

%% file: outro.tex

\section{Concluding Remarks} \label{sec:outro}

In this manuscript, we study robust MDPs on continuous state spaces with a structured $\mathrm{s}$-rectangular ambiguity set. 
The proposed ambiguity set lies within the convex hull of unknown generating kernels, 
which are only accessible via samples.
We propose a stochastic first-order method for robust policy evaluation, and establish its high-probability convergence to the robust value function, from which we establish an $\tilde{\cO}(1/\epsilon^2)$ sample complexity.
With the high-probability accuracy certificate of the robust policy evaluation, we introduce an approximate policy iteration method that finds an $\epsilon$-optimal policy for the robust policy optimization problem with $\tilde{\cO}(1/\epsilon^2)$ sample complexity. 
Notably the proposed methods operate independently of the size of the state space, and thus can be implemented efficiently for continuous state spaces. 
The obtained sample complexities also appear to be new for solving robust MDPs with continuous state spaces.

We discuss a few directions that might be worth future investigation. 
First, it is interesting to consider ambiguity sets that are, in some sense, generated nonlinearly from the generating kernels, \emph{e.g.}, when the generation process involves a nonlinear parameterized function that is convex in its parameters. 
It would also be interesting to study whether similar methods could be developed for nonparametric ambiguity sets, for instance those generated by $\phi$-divergence or Wasserstein distance. 
Lastly, instead of the approximate policy iteration considered in this manuscript, it is also highly rewarding to directly design primal-dual-based methods for robust MDPs in the stochastic setting.

%% file: appendix.tex
\appendix
\section{Appendix}\label{sec_appendix}
\revm{We introduce the following non-robust Bellman evaluation operator, namely for all $v\in C_b(\cS)$ and $s\in\cS$,
\begin{align*}
    \cT^P_\pi v(s) = \sum_{a\in\cA} \pi(a|s) \sbr{r(s,a) + \gamma \int_{\cS} v(s') P(\diff s'|s,a)}.
\end{align*}
We also define the robust Bellman evaluation operator through
\begin{align*}
    \cT_\pi v(s) = \inf_{P\in\cP} \cT^P_\pi v(s),
\end{align*}}
and finally, the robust Bellman operator $\cT$ is defined through
\begin{equation*}
    \cT v(s)= \revm{\sup_{\pi\in\Pi}\cT_\pi v(s). }
\end{equation*}
\revm{Lemma~\ref{lemma-exist-opt-policy} shows that if $\cP$ is a mixture ambiguity set in the sense of Definition~\ref{def_mix_amb_set}, then the robust Bellman operator admits a fixed point and the robust MDP~\eqref{def_rmdp_general} is solvable. The corresponding results for tabular robust MDPs are established in~\cite{wiesemann2013robust, le2007robust}}. 
In addition, Lemma~\ref{lemma-exist-opt-policy}\ref{item-dp-1} and~\ref{item-dp-2} use standard proof techniques similar to those in~\citep[Theorem 2.7]{neufeld2023markov}, which studies $(\rm s,\rm a)$-rectangular {ambiguity} sets \cite{nilim2005robust, iyengar2005robust} on measurable state and action spaces.
\revm{In the ensuing discussion, we denote $C_b(\cS)$ the set of continuous and bounded functions over $\cS$, and note that $C_b(\cS)$ equipped with the supremum norm is a Banach space.}
\begin{lemma} \label{lemma-exist-opt-policy}
The following hold for the robust MDP problem \eqref{def_rmdp_general} with mixture ambiguity set (Definition~\ref{def_mix_amb_set}).
\begin{enumerate}[label = (\roman*)]
    \item \label{item-dp-1} For every $v\in C_b(\cS)$, there exists $P'\in\cP$ such that for all $(s,p)\in\cS\times\Delta_{\cA}$, we have $\cT^{P'}_\pi v(s)=\cT_\pi v(s)$.
    Moreover, there exists $\pi'\in\Pi$ such that for every $s\in\cS$, we have 
    $\cT_{\pi'} v(s) = \cT v(s).$
    Furthermore, 
    there exists a pair $(P',\pi')\in\cP\times\Pi$ that $\cT^{P'}_{\pi'} v(s)=\cT v(s).$
    \item \label{item-dp-2} For any $v\in C_b(\cS)$, 
    we have $\cT v\in C_b(\cS)$. In addition, for all $v,w\in C_b(\cS)$, we have $\|\cT v - \cT w\|_\infty \le \gamma \|v-w\|_\infty$. In particular, there exists a unique $v\opt\in C_b(\cS)$ such that $\cT v\opt= v\opt$. Moreover, for every $v_0\in C_b(\cS)$, we have $v\opt=\lim_{n\to\infty}\cT^n v_0.$
    \item \label{item-dp-3} 
    Let $(P\opt,\pi\opt)$ be such that $\cT^{P\opt}_{\pi\opt} v\opt(s)=\cT v\opt(s).$
    We then have $v\opt(s)=V^{P\opt}_{\pi\opt}(s)= \sup_{\pi\in\Pi}\inf_{P\in\cP} V^P_\pi(s)$ for all $s\in\cS$.
\end{enumerate}
\end{lemma}

\begin{proof}
Let $v\in C_b(\cS)$, and define 
the map  $f: \cS\times\Delta_{\cA}\times\cP_s \to \RR$ through
\begin{align*}
f(s,p,Q) = \sum_{a\in\cA} p(a) \sbr{r(s,a) + \gamma\int_{\cS} v(s') Q(\diff s'|s,a)} \quad \forall Q\in\cP_s.
\end{align*}
We aim to show that $f$ is continuous. To this end, let \revm{$(s,p,Q)\in\cS\times\Delta_{\cA}\times\cP_s $ and $\{(s_n, p_n, Q_n)\}_{n\in\NN}$ be a sequence in $\cS\times\Delta_{\cA}\times\cP_s$ with $(s_n, p_n, Q_n) \to (s,p,Q)$ as $n\to\infty$, where $Q_n\to Q$ is understood as $Q_n(\cdot|s,a)$ converges weakly to $Q(\cdot|s,a)$ for every $(s,a)\in\cS\times\cA$.}
We then have
\begin{align*}
&|f(s_n,p_n,Q_n)-f(s,p,Q)| \\&\le \revm{|f(s_n,p,Q_n)-f(s,p,Q_n)|}+ |f(s_n,p_n,Q_n)-f(s_n,p,Q_n)|  + | f(s,p,Q_n)-f(s,p,Q)|,
\end{align*}
\revm{where the first term converges to $0$ because $r$ is continuous in $s$ and because $\int_\cS v(s') Q_n(\diff s'|s,a)$ is continuous in $s$, which in turn holds because $Q_n(\diff s'|s,a)$ is continuous in $s$. The second term converges to $0$ because~$r$ and $v$ are bounded and $p_n$ converges to $p$.
The third term converges to zero because $Q_n(\cdot|s,a)$ converges weakly to $Q(\cdot|s,a)$ and the integrand is continuous.}
Hence,~$f$ is continuous. 
Now, we define the $\argmin$ correspondence $\mu: \cS\times\Delta_{\cA} \rightrightarrows \cP_s$ through
\[\mu(s,p)=\{Q\in\cP_s:  f(s,p,Q)=g(s,p)\}\]
where 
\[g(s,p)= \min_{Q\in\cP_s} f(s,p,Q).\]
The set-valued mapping $s\mapsto \cP_s$ is continuous by Lemma~\ref{lemma:cP_s-continuous}, and hence Berge's maximum theorem~\citep[Theorem~17.31]{aliprantis2006infinite} implies that $g$ is continuous and that $\mu$ is non-empty-valued and upper hemicontinuous.
Measurable selection theorem~\citep[Theorem~18.19]{aliprantis2006infinite} then asserts that $\mu(s,p)$ has a Borel-measurable selector~$Q'$. \revm{By construction, we have $Q'\in\cP_s$ because $Q'$ is a Borel-measurable function from $\cS\times\Delta(\cA)$ to $\Delta(\cS)$ for any fixed $a\in\cA$.
\citep[Theorem~15.13]{aliprantis2006infinite} then implies that $s\mapsto Q'(\cB|s,p,a)$ is a Borel-measurable real-valued function for every Borel set $\cB\subseteq \mathbb R^S$.
Since compositions of Borel measurable functions remain Borel measurable, for any fixed $\pi\in\Pi$, then the kernel~$P'$ defined through $P'(\cdot|s,a)=Q'(\cdot|s,\pi(\cdot|s),a)$ is an element of $\cP$.
We then have 
\[\cT^{P'}_\pi v(s)= \inf_{Q\in\cP_s} f(s,p,Q)=\inf_{P\in\cP} \cT^P_\pi v(s),\]
where the first equality follows from the previous construction that $P'(\cdot|s,a)=Q'(\cdot|s,\pi(\cdot|s),a)$, and the second equality holds because $\cP$ is $\rm s$-rectangular.}
Similarly, we define the $\argmax$ correspondence $ \nu: \cS \rightrightarrows \Delta_{\cA}$ through
\[
\nu(s)=\{p \in \Delta_{\cA}: g(s,p)=h(s)\} ,
\]
where
\[h(s)=\max_{p\in\Delta_{\cA}} g(s,p).\]
Since $g$ is continuous and the constant correspondence $\cS\rightrightarrows\Delta_{\cA}$ is compact-valued therefore trivially continuous, by Berge's maximum theorem~\citep[Theorem~17.31]{aliprantis2006infinite}, $h$ is continuous and $\nu$ is non-empty-valued and upper hemicontinuous.
The measurable selection theorem~\citep[Theorem~18.19]{aliprantis2006infinite} thus implies that $\nu(s)$ has a Borel-measurable selector $p'$. 
\revm{A similar argument leading up to the claim that $P'\in\cP$ implies that the policy~$\pi'$ defined through $\pi'(\cdot|s)=p'(\cdot|s)$ is an element of $\Pi$.}
Then, we have
\begin{align*}
    \cT v(s)=\sup_{\pi\in\Pi}\cT_\pi v(s)
    =\max_{p\in\Delta_{\cA}} g(s,p)
    =g(s,\pi'(\cdot|s))
    &= \inf_{P\in\cP} \cT^P_{\pi'} v(s) = \cT^{P'}_{\pi'} v(s),
\end{align*}
\revm{where the first equality holds by definition of $\cT$, the second equality follows because of the separability of~$\Pi$ across $s$, the third equality follows from the previous construction where $\pi'(\cdot|s)=p'(\cdot|s)$, and the fourth equality uses the definition of $g$. The last equality then follows by construction of $P'$.}
Assertion~\ref{item-dp-1} thus follows.
For Assertion~\ref{item-dp-2}, note first that $\cT v$ inherits continuity from the continuity of $h$ shown in the proof of Assertion~\ref{item-dp-1}. Since $v\in C_b(\cS)$, we have
\begin{align*}
    \|\cT v\|_\infty\le \max_{a\in\cA} \|r(\cdot,a)\|_\infty + \gamma \|v\|_\infty,
\end{align*}
which readily implies that $\cT v \in C_b(\cS)$. 
Moreover, we have
\begin{align*}
    |\cT v- \cT w| \le \gamma \sup_{p\in\Delta_{\cA}}\sup_{P\in\cP}\left|\int_{\cS} (v(s')-w(s')) \sum_{a\in\cA} P(\diff s'|s,a) p(a)\right| \le \gamma \|v-w\|_\infty,
\end{align*}
where the first inequality follows by the definition of $\cT$ and the observation that for any set $\cQ$ and functions $G,H:\cQ\to\RR$, we have
\begin{equation*}
\Big|\inf _{q\in \cQ} G(q)-\inf _{q \in \cQ} H(q)\Big| \leq \sup _{q \in \cQ}|G(q)-H(q)| .
\end{equation*}
We have thus shown that $\cT$ is a contraction operator on $C_b(\cS)$. Banach's fixed point theorem then \revm{applies because $\cC_b(\cS)$ equipped with supremum norm is a Banach space. Hence,}
there exists a unique fixed point $v\opt\in C_b(\cS)$ of $\cT$ such that $v\opt=\cT v\opt=\lim_{n\to\infty} \cT^n v_0$ for all $v_0\in C_b(\cS)$. Hence, Assertion~\ref{item-dp-2} follows.
For Assertion~\ref{item-dp-3}, 
note that
\begin{align*}
    v\opt(s)=\cT v\opt(s) \ge \cT_\pi v\opt(s) \ge \lim_{n\to\infty} (\cT_\pi)^n v\opt(s) = V_\pi(s) \quad\forall \pi\in\Pi, \ s\in\cS,
\end{align*}
where the first inequality holds by definitions of $\cT$ and $\cT_\pi$, the second inequality follows because $\cT_\pi$ is monotone, and the second equality uses the contraction property of $\cT_\pi$. The above expression readily implies
\begin{align*}
    v\opt(s) \ge \sup_{\pi\in\Pi} V_\pi(s)= \sup_{\pi\in\Pi} \inf_{P\in\cP} V^P_\pi(s)\quad \forall s\in\cS.
\end{align*}
On the other hand, we have
\begin{align*}
    v\opt(s)=\cT^{P\opt}_{\pi\opt} v\opt(s) \le \cT^{P}_{\pi\opt} v\opt(s) \le \lim_{n\to\infty} (\cT^{P}_{\pi\opt})^n v\opt(s)=V^P_{\pi\opt}(s)\quad \forall P\in\cP,\ s\in\cS,
\end{align*}
where the inequality uses the optimality of~$P\opt$, the second inequality holds because $\cT^{P}_{\pi\opt} $ is monotone, and the second inequality follows from the contraction property of $\cT^P_{\pi\opt}$.
We have thus shown that $v\opt(s)\le V^P_{\pi\opt}(s)$ for all $P\in\cP$ and $s\in\cS$, which in turn implies that \[v\opt(s)=\inf_{P\in\cP}V^P_{\pi\opt}(s) \le \sup_{\pi\in\Pi}\inf_{P\in\cP}V^P_{\pi}(s)\quad\forall s\in\cS.\]
The claim follows by combining the previous insights \[\sup_{\pi\in\Pi} \inf_{P\in\cP} V^P_\pi(s)\le v\opt(s)\le \sup_{\pi\in\Pi} \inf_{P\in\cP} V^P_\pi(s) \quad \forall s\in\cS\]
and observing that 
$v\opt(s)=\cT^{P\opt}_{\pi\opt} v\opt(s) = \lim_{n\to\infty}(\cT^{P\opt}_{\pi\opt})^n v\opt(s) = V^{P\opt}_{\pi\opt}(s)$ for all  $s\in\cS$.

\end{proof}
\revm{
\begin{lemma} \label{lemma:cP_s-continuous}
Let $\cS \subseteq \mathbb{R}^d$ be a closed set and define the set-valued mapping $F : \cS \rightrightarrows (\Delta_{\cS})^A$ through $F(s)=\cP_s$ where $\cP_s$ is specified in Definition~\ref{def_mix_amb_set}. Then $F$ is a continuous correspondence (\emph{i.e.,} both upper and lower hemicontinuous) with respect to the product topology of the weak topologies on $\Delta_{\cS}$. 
\end{lemma}
\begin{proof}
Define the function $\Phi : \cS \times W \to (\Delta_{\cS})^A$ through $\Phi(s,\mathrm{w}) = \sum_{k=1}^K \mathrm{w}_k P_k(\cdot|s,\cdot).$
Since addition and scalar multiplication of probability measures are continuous under the weak topology, and each $P_k(\cdot|s,a)$ is continuous in $s$ for all $a\in\cA$ according to Definition~\ref{def_mix_amb_set}, the function $\Phi$ is jointly continuous under the product topology.
By construction, we have $F(s) = \Phi(\{s\} \times W)$.
We first show upper hemicontinuity of $F.$
Fix $s \in \cS$ and let $U \subseteq \Delta_{\cS}$ be open with $F(s) \subseteq U$.
For each ${\mathrm{w}} \in W$, continuity of $\Phi$ yields open neighborhoods $V_{\mathrm{w}}$ of~$s$ in $\cS$ and $X_{\mathrm{w}}$ of ${\mathrm{w}}$ in $\Delta_{\cK}$ such that
$\Phi(V_{\mathrm{w}} \times X_{\mathrm{w}}) \subseteq U.$
The family $\{X_{\mathrm{w}}\}_{{\mathrm{w}} \in W}$ is an open cover of the compact set $W$, so there exist ${\mathrm{w}}^{(1)},\dots,{\mathrm{w}}^{(m)}$ with
$W \subseteq \bigcup_{i=1}^m X_{{\mathrm{w}}^{(i)}}$.
Let $V = \bigcap_{i=1}^m V_{{\mathrm{w}}^{(i)}}$. Then $V$ is a neighborhood of~$s$ and for all $s' \in V$,
\[
F(s') = \Phi(s'\times W) \subseteq U,
\]
which proves upper hemicontinuity.
We now proceed to show lower hemicontinuity of $F$.
Fix $s \in \cS$ and let $U \subseteq \Delta_{\cS}$ be open with $F(s) \cap U \neq \emptyset$.
Choose $\mu \in F(s) \cap U$. Then $\mu = \Phi(s,{\mathrm{w}})$ for some ${\mathrm{w}} \in W$.
By continuity of $\Phi$, there exists a neighborhood $V$ of $s$ such that
$\Phi(s',{\mathrm{w}}) \in U$ for all $s' \in V.$
Hence, for all $s' \in V$,
\[
F(s') \cap U \neq \varnothing,
\]
which establishes lower hemicontinuity.
Therefore, $F$ is continuous as a correspondence.
\end{proof}
}

\begin{lemma}
\label{lemma-fixed-prop34}

Let $(\Omega,\cF,\PP)$ be a probability space and $\{\cF_m\}_{m\ge 0}$ be a filtration. Assume that the stochastic processes $\{X_m\}_{m\ge 0}$ and $\{\hat X_m\}_{m\ge0}$ with $X_m:\Omega\to[0,\bar{X}]$, $\hat X_m:\Omega\to\R$ are such that $X_m$ and~$\hat X_m$ are both $\cF_m$-measurable for every $m\ge 0$. 
Fix a positive integer $M>1$, and suppose that for any $\delta\in(0,1)$, there exist $\underline\varepsilon,B,J>0$ 
such that for all $m\in[M]$,
\[\|\EE[ \hat X_m | \mac F_{m-1}] - X_m \|_\infty \le \underline\varepsilon, \quad  \| \hat X_m \|_\infty \le B, \quad \EE\Big[ \big\|\hat X_m - X_m\big\|_\infty^2 \mid \cF_{m-1}\Big] \le J\]
hold with probability $1-\delta/({4}M)$.
Then, for any $\varepsilon\ge\underline\varepsilon$ and $\vartheta \in \Delta_{[M]}$, we have with probability $1-\delta{/2}$ that
\[\displaystyle\left|\sum_{m=1}^M \vartheta_m(\hat X_m-X_m)\right| \leq \varepsilon + (B+\bar{X}) \sqrt{2 \sum_{m=1}^M \vartheta_m^2 \log \left(\frac{8 M}{\delta }\right)} + \sqrt{\frac{J\delta}{M}}.  \]
\end{lemma}

\begin{proof}
Define a random sequence $\{Y_m\}_{m=0}^M$ through $Y_m = \hat X_m-X_m$ for all $m\in[M]$, and let $Y'_m=Y_m \mathds{1}_{\{\|Y_m\|_\infty \le B + \bar{X}\}}$ for all $m\in[M]$ and $s\in\Omega$. 
We aim to establish an upper bound for $\EE[Y'_m]$.
To this end,
observe first that
\begin{equation}
\begin{split}\label{eq-abs-Ym}
    \EE\Big[ \big|Y_m \mathds{1}_{\{\|Y_m\|_\infty {>} B+\bar{X}\}}\big|\Big] \le \sqrt{\EE[ \|Y_m\|^2] \cdot \EE\Big[\big|\mathds{1}_{\{\|Y_m\|_\infty {>} B+\bar{X}\}}\big|^2\Big] }
    \le \sqrt{\frac{J\delta}{{4}M}},
\end{split}
\end{equation}
where the first inequality follows from Cauchy-Schwarz inequality, and the second inequality holds because of the assumptions $\EE[ \|\hat X_m - X_m\|^2 \mid \cF_{m-1}] \le J$ together with  $\PP(\| \hat X_m \|_\infty \le B) \ge 1-\delta/({4}M)$.
Next, note that
\begin{equation}
\begin{split}\label{eq-Ym}
    \EE[Y_m]&=\EE\Big[ \big(\hat X_m - X_m \big)\mathds{1}_{\cbr{\|\EE[ \hat X_m | \mac F_{m-1}] - X_m \|_\infty \le \underline\varepsilon}} + \big(\hat X_m - X_m \big)\mathds{1}_{\cbr{\|\EE[ \hat X_m | \mac F_{m-1}] - X_m \|_\infty > \underline\varepsilon}}\Big] 
    \\&\le \EE\Big[ \big\|\EE[ \hat X_m | \mac F_{m-1}] - X_m \big\|_\infty \mathds{1}_{\cbr{\|\EE[ \hat X_m | \mac F_{m-1}] - X_m \|_\infty \le \underline\varepsilon}}\Big] 
    \\&\quad+ \sqrt{\EE\Big[\big\|\hat X_m  - X_m \big\|_\infty^2\Big] \PP\Big(\|\EE[ \hat X_m | \mac F_{m-1}] - X_m \|_\infty > \underline\varepsilon\Big) }
    \\&\le \underline\varepsilon+ \sqrt{\frac{J\delta}{{4}M}},
\end{split}
\end{equation}
where the first inequality holds by the law of iterated expectations and the Cauchy-Schwarz inequality, and the second inequality follows from the assumptions $\EE[ \|\hat X_m - X_m\|^2 \mid \cF_{m-1}] \le J$ together with  $\PP(\| \hat X_m \|_\infty \le B) \ge 1-\delta/({4}M)$.
We then have for every $m\in[M]$ that
\begin{align*}
\EE[Y'_m] =\EE \Big[  Y_m - Y_m \mathds{1}_{\{\|Y_m\|_\infty {>} B+\bar{X}\}}  \Big] 
\le \EE\Big[ Y_m \Big] +  \EE\Big[ \big|Y_m \mathds{1}_{\{\|Y_m\|_\infty {>} B+\bar{X}\}}\big|\Big] 
\le \underline\varepsilon + \sqrt{\frac{J\delta}{M}} ,
\end{align*}
where the second inequality follows by combining the previous two observations~\eqref{eq-abs-Ym} and~\eqref{eq-Ym}.
{Since the sequence $\{\vartheta_m (Y'_m-\EE[Y'_m|\cF_{m-1}])\}_{m\ge0}$ forms a martingale difference sequence with respect to $\{\cF_m\}_{m\ge0}$, we may then invoke Azuma's inequality~\citep[Proposition~2.20]{wainwright2019high} to obtain}
\begin{equation*}
\begin{split}
    \left|\sum_{m=1}^M \vartheta_m Y'_m\right| 
    &\le (B+\bar{X})\sqrt{2 \sum_{m=1}^M \vartheta_m^2 \log \left(\frac{{8}M}{\delta}\right)} + \underline\varepsilon + \sqrt{\frac{J\delta}{M}},
\end{split}
\end{equation*}
which holds with probability $1-\delta/{4}$.
Finally, since $\PP(\| \hat X_m \|_\infty \le B) \ge 1-\delta/({4}M)$, we obtain
\[\PP\rbr{\sum_{m=1}^M Y_m\ne \sum_{m=1}^M Y'_m} \le \frac{\delta}{{4}} .\]
Hence, the claim follows.
\end{proof}

\begin{lemma}[Pinsker's inequality {\citep[Lemma~2.5(i)]{tsybakov2008nonparametric}}]\label{lemma:pinsker}
Let $p,q\in\Delta_{[K]}$. We then have $2\mathsf{D}(p,q)\ge \|p-q\|_1^2 $.
\end{lemma}